\documentclass[12pt,english]{amsart}
\usepackage{ulem} 
\calclayout
\usepackage{float}
\usepackage{amsmath,amsfonts,amssymb,amsrefs}
\usepackage{stmaryrd}
\usepackage{graphicx, epsfig, psfrag}
\usepackage{epstopdf}
\usepackage{color}
\usepackage{enumerate}
\usepackage{mathrsfs}%
\usepackage{algorithm}
\usepackage[noend]{algpseudocode}

\newtheorem{rem}{Remark}
\theoremstyle{plain}
\newtheorem{theorem}{Theorem}
\newtheorem{lemma}[theorem]{Lemma}
\newtheorem{innercustomlemma}{Lemma}
\newenvironment{customlemma}[1]
  {\renewcommand\theinnercustomlemma{#1}\innercustomlemma}
  {\endinnercustomthm}

\newtheorem{innercustomthm}{Theorem}
\newenvironment{customtheorem}[1]
  {\renewcommand\theinnercustomthm{#1}\innercustomthm}
  {\endinnercustomthm}
\newcommand\A{\mathcal{A}}
\newcommand\HH{{\mathscr{H}}}
\newcommand\F{\mathcal F}
\newcommand\T{\mathcal T}
\newcommand\RR{\mathbb R}
\newcommand\NN{\mathbb N}

\newcommand\PP{\mathbb P}

\newcommand\bg{{g}}

\newcommand\nn{{\boldsymbol n}}

\newcommand\xx{{\boldsymbol x}}

\newcommand\uurm{u^0}

\newcommand\vv{{\boldsymbol v}}
\newcommand\vvrm{v^0}
\newcommand\I{\mathcal{I}}

\newcommand\LRM{\Lambda^0} 
\newcommand\LTRM{{\widetilde\Lambda}_h^0}

\newcommand\bsigma{{\boldsymbol\sigma}}
\newcommand\btau{{\boldsymbol\tau}}
\newcommand{\salto}[1]{{{\llbracket} #1 {\rrbracket}}}

\newcommand{\trih}{\mathcal{T}_h}
\newcommand{\triH}{{\mathcal{T}_H}}

\newcommand{\btriH}{{\partial\mathcal{T}_H}}

\newcommand{\faceh}{\mathcal{F}_h}

\newcommand\dO{\partial\Omega}
\newcommand\dtau{{\partial\tau}}
\newcommand\OO{\Omega}

\newcommand\astab{\alpha_{\rm{stab}}}
\newcommand\amin{a_{\text{min}}}
\newcommand\amax{a_{\text{max}}}

\newcommand\Hdiv{H({\rm{div};\OO)}}
\newcommand\HdivA{H_\A({\rm{div};\OO)}}

\newcommand\Hdivtau{H({\rm{div};\tau)}}

\newcommand\VVRM{\PP^0(\triH)}
\newcommand\tildeLambda{{\widetilde\Lambda}}

\newcommand\lambdarm{\lambda^0}
\newcommand\tlambda{\tilde\lambda}
\newcommand\tlambdarmj{\tilde\lambda^{0,j}}

\newcommand\tlambdarm{\tilde\lambda^0}
\newcommand\murm{{\mu^0}}
\newcommand\tmurm{\tilde\mu_h^0}

\newcommand\VH{{H^1(\triH)}}
\newcommand\tildeVV{{\widetilde H^1(\triH)}}
\newcommand\LH{\Lambda(\triH)}
\newcommand\Lh{\Lambda_h}

\def\div{\operatorname{div}}

\def\supp{\operatorname{supp}}
\def\bgrad{\operatorname{\boldsymbol{\operatorname{\nabla}}}} 
\def\spann{\operatorname{span}} 

\definecolor{bluegreen}{rgb}{0,0.75,0.75}

\def\Nf{\rm{N^f}}
\global\long\def\sumtautH{\sum_{\tau\in\triH}}

\begin{document}
\title{Hybrid Localized Spectral Decomposition for multiscale problems}
\date{November 30, 2020}
\author{Alexandre L. Madureira}
\address{
Laborat\'orio Nacional de Computa\c c\~ao Cient\'\i fica, Petr\'opolis - RJ, Brazil} 
\address{
Funda\c c\~ao Get\'ulio Vargas, Rio de Janeiro - RJ, Brazil} 
\email{alm@lncc.br, alexandre.madureira@fgv.br}
\author{Marcus Sarkis}
\address{
Mathematical Sciences Department
Worcester Polytechnic Institute, USA}%
\email{msarkis@wpi.edu}
\begin{abstract}
We consider a finite element method for elliptic equations with heterogeneous and possibly high-contrast coefficients based on primal hybrid formulation. We assume minimal regularity of the solutions. A space decomposition as in FETI and BDCC induces an embarrassingly parallel pre-processing and leads to a final system of size independent of the coefficients. The resulting solution is in equilibrium, and all PDEs involved are elliptic. One of the problems in the pre-processing step is non-local but with exponentially decaying solutions, enabling a practical scheme where the basis functions have an extended, but still local, support. 

To make the method robust with respect to high-contrast coefficients, we enrich the space solution  via local eigenvalue problems, obtaining optimal a priori error estimate that mitigates the effect of having coefficients with different magnitudes. The technique developed is dimensional independent and easy to extend to other elliptic problems such as elasticity. 
\end{abstract}
\maketitle

\section{Introduction}\label{s:intro}\footnote{Some of the methods analyzed here have been 
previously discussed in the 8-page proceeding paper~\cite{MR3989895}.}
Consider the problem of finding the weak solution $u:\OO\to\RR$ of 
\begin{equation}\label{e:elliptic}
\begin{gathered}
-\div\A\bgrad u=g\quad\text{in }\OO,
\\
u=0\quad\text{on }\partial\OO,
\end{gathered}
\end{equation}
where $\OO\subset\RR^d$ with $d=2$ or $3$ for simplicity, and is an open bounded domain with polyhedral boundary $\partial\Omega$, the symmetric tensor $\A\in[L^\infty(\OO)]_{\text{sym}}^{d\times d}$ is uniformly positive definite and bounded, and $g$ is part of the given data. 

It is hard to approximate such problem in its full generality using numerical methods, in particular because of the low regularity of the solution and its multiscale behavior. Most convergent proofs either assume extra regularity or special properties of the coefficients~\cites{AHPV,MR3050916,MR2306414,MR1286212,babuos85,MR1979846,MR2058933,HMV,MR1642758,MR3584539,MR2030161,MR2383203,vs1,vs2,MR2740478}. Some methods work even considering that the solution has low regularity~\cites{MR2801210,MR2753343,MR3225627,MR3177856,MR2861254} 
but are based on ideas that differ considerably from what we advocate here
and do not cover in depth the high-contrast case. 

As in many multiscale methods previously considered, our starting point is the decomposition of the solution space into \emph{fine} and \emph{coarse} spaces that are adapted to the problem of interest. The exact definition of some basis functions requires solving global problems, but, based on decaying properties, only local computations are required, although these are not restricted to a single element. It is interesting to notice that, although the formulation is based on hybridization, the final numerical solution is defined by a sequence of elliptic problems. 

The idea of using exponential decay to localize global problems was already considered by the interesting approach developed under the name of \emph{Localized Orthogonal Decomposition} (LOD)~\cites{MR2831590,MR3591945,MR3246801,MR3552482} which are
related to ideas of \emph{Variational Multiscale Methods}~\cites{MR1660141,MR2300286}. In their case, convergence follows from a special orthogonality property.

One difficulty that hinders the development of efficient methods is the presence of high-contrast coefficients~\cites{MR3800035,MR2684351,MR2753343,MR3704855,MR3225627,MR2861254}. When LOD or VMS methods are considered, high-contrast coefficients might slow down the exponential decay of the solutions, making the method not so practical. Here in this paper, in the presence of rough coefficients, spectral techniques are employed to overcome such hurdle, and by solving local eigenvalue problems we define a space where the exponential decay of solutions is insensitive to high-contrast coefficients. Additionally, the spectral techniques remove 
macro-elements corner singularities that occur in LOD methods based on
mixed finite elements. We note the proposal in~\cite{CHUNG2018298} of generalized multiscale finite element methods based on eigenvalue problems inside the macro elements, with basis functions with support weakly dependent of the log of the contrast. Here, we propose eigenvalue problems based on edges of macro element removing the dependence
of the contrast. 

We now further detail the problem under consideration. For almost all $\xx\in\OO$ let the positive constants $\amin$ and $\amax$ be such that 
\begin{equation}\label{e:bounds}
\amin|\vv|^2
\le a_-(\xx)|\vv|^2\le\A(\xx)\,\vv\cdot\vv
\le a_+(\xx)|\vv|^2\le\amax|\vv|^2
\quad\text{for all }\vv\in\RR^d,
\end{equation}
where $a_-(\xx)$ and $a_+(\xx)$ are the smallest and largest eigenvalues of 
$\A(\xx)$.

The remainder of the this paper is organized as follows. Section~\ref{Hybrid} describes a suitable primal hybrid formulation for the problem~\eqref{e:elliptic}, which is followed in Section~\ref{s:Hlocal} by its a discrete formulation. A discrete space decomposition is introduced to transform the discrete saddle-point problem into a sequence of elliptic discrete problems. The analysis of the exponential decay of the multiscale basis function is considered in Section~\ref{ss:lowcon}. To overcome the possible deterioration of the exponential decay for high-contrast coefficients, in Section~\ref{ss:dhc} the \emph{Localized Spectral Decomposition} (LSD) method is designed and fully analyzed. To allow an efficient pre-processing numerical scheme, Section~\ref{ss:findim} discusses how to reduce the right-hand side space dimension without losing a target accuracy, and also develops $L^2(\Omega)$ a priori error estimates. Section~\ref{s:Algorithms} gives a global overview of the LSD \emph{algorithm} proposed. Appendix~\ref{s:Auxiliaryresults} provides some mathematical tools and Appendix~\ref{s:Notations} refers to a notation library for the paper. 

\section{Continuous Problem using Hybrid Formulation} \label{Hybrid} 
We start by recasting the continuous problem in a weak formulation that depends on a polyhedral regular mesh $\triH$ and let $\F_H$ be the set of faces on $\triH$ --- we assume that each element has diameter smaller than $H$, and at most $\Nf$ faces. \label{p:triHdef} Without loss of generality we adopt above and in the remainder of the text the terminology of three-dimensional domains, denoting for instance the boundaries of the elements by faces. For a given element $\tau\in\triH$ let $\dtau$ denote its boundary and $\nn^\tau$ the unit size normal vector that points outward $\tau$. \label{p:nndef}
We denote by $\nn$ the outward normal vector on $\dO$. Consider now the following spaces:
\begin{equation}\label{e:spaces}
\begin{gathered}
  \VH=\{v\in L^2(\OO):\,v|_\tau\in H^1(\tau),\,\tau\in\triH\},
  \\
  \LH=\biggl\{\prod_{\tau\in\triH}\btau\cdot\nn^\tau|_{\dtau}:\,\btau\in\Hdiv\biggr\}
  \subsetneq\prod_{\tau\in\triH}H^{-1/2}(\dtau).
\end{gathered}
\end{equation}
For $w$, $v\in \VH$ and $\mu\in\LH$ define 
\begin{gather*}\label{p:inner}
  (w,v)_{\triH}=\sum_{\tau\in\triH}\int_\tau wv\,d\xx, 
  \qquad
  (\mu,v)_{\btriH}=\sum_{\tau\in\triH} (\mu,v)_{\dtau},
\end{gather*}
where $(\cdot,\cdot)_{\dtau}$ is the dual product involving $H^{-1/2}(\dtau)$ and $H^{1/2}(\dtau)$. Then
\[
(\mu,v)_{\dtau}=\int_\tau\div\bsigma v\,d\xx+\int_\tau\bsigma\cdot\bgrad v \,d\xx
\]
for all $\bsigma\in\Hdivtau$ such that $\bsigma\cdot\nn^\tau=\mu$. We also define 
\begin{equation}\label{e:norms}
\begin{gathered}
\|\bsigma\|_{\HdivA}^2=\|\A^{-1/2}\bsigma\|_{0,\Omega}^2+H^2\|\div\bsigma\|_{0,\Omega}^2, 
\\ 
|v|_{H_\A^1(\triH)}^2=\sumtautH\|\A^{1/2}\bgrad v\|_{0,\tau}^2, 
\qquad
\|\mu\|_{H_\A^{-1/2}(\triH)}=\inf_{\substack{\bsigma\in\Hdiv\\\bsigma\,\nn^\tau=\mu\text{ on } \dtau,\,\tau\in\triH}}\|\bsigma\|_{\HdivA}.
\end{gathered}
\end{equation}
We use analogous definitions on subsets of $\triH$, in particular when the subset consists of a single element $\tau$ (and in this case we write $\tau$ instead of $\{\tau\}$). We note that since $\amin$ is positive and $\amax$ is bounded, then $\|\cdot\|_{\HdivA}$ and $|\cdot|_{H_\A^1(\triH)}$ are equivalent to the usual norms $\|\cdot\|_{\Hdiv}$ and $|\cdot|_{H^1(\triH)}$. 

In the primal hybrid formulation~\cite{MR0431752}, $u\in\VH$ and $\lambda\in\LH$ are such that 
\begin{equation}\label{e:weak-hybridH}
\begin{alignedat}{3}
&(\A\bgrad u,\bgrad v)_{\triH}-(\lambda,v)_{\btriH}&&=(g,v)_{\triH}\quad&&\text{for all }v\in\VH,
\\
&(\mu,u)_{\btriH}&&=0\quad&&\text{for all }\mu\in\LH.
\end{alignedat}
\end{equation}
Following~\cite{MR0431752}*{Theorem 1}, it is possible to show that the solution $(u,\lambda)$ of~\eqref{e:weak-hybridH} is such that $u\in H_0^1(\Omega)$ satisfies~\eqref{e:elliptic} in the weak sense and $\lambda=\A\bgrad u\cdot\nn^\tau$ for all elements $\tau$. 

In the spirit of the Multiscale Hybrid Methods~\cites{AHPV,HMV,HPV,MR3584539} and FETI methods~\cites{FETI,totalfeti,MR1285024,MR2282408,MR2104179}, we consider the decomposition
\begin{equation*}\label{e:VHdecomp}
\VH=\VVRM\oplus\tildeVV, 
\end{equation*}
where $\VVRM$ is the space of piecewise constants, and $\tildeVV$ is its $L^2(\tau)$ orthogonal complement, i.e., the space of functions with zero average within each element $\tau\in\triH$: 
\begin{equation}\label{e:poHtdef}
\begin{gathered}
  \VVRM=\{v\in\VH:\,v|_\tau\text{ is constant, }\tau\in\triH\},
  \\
  \tildeVV=\{\tilde v\in\VH:\,\int_\tau\tilde v\,d\xx=0,\,\tau\in\triH\}, 
  \qquad
  \widetilde H^1(\tau)=\{\tilde v|_\tau:\,\tilde v\in\VH\}. 
\end{gathered}
\end{equation}
We then write $u=\uurm+\tilde u$, where $\uurm\in\VVRM$ and $\tilde u\in\tildeVV$, and find from~\eqref{e:weak-hybridH} that 
\begin{equation}\label{e:weak-e1}
\begin{alignedat}{3}
  &(\lambda,\vvrm)_{\btriH}&&=-(g,\vvrm)_{\triH}\quad&&\text{for all }\vvrm\in\VVRM,
  \\
  &(\mu,\uurm+\tilde u)_{\btriH}&&=0\quad&&\text{for all }\mu\in\LH,
\end{alignedat} 
\end{equation}
and that
\begin{equation}\label{e:weak-e3}
  (\A\bgrad\tilde u,\bgrad\tilde v)_{\triH}
  =(\lambda,\tilde v)_{\btriH}+(g,\tilde v)_{\triH}
  \quad\text{for all }\tilde v\in\tildeVV. 
\end{equation}
Let $T:\LH\to\tildeVV$ and $\tilde T:L^2(\Omega)\to\tildeVV$ be such that, for $\mu\in\LH$, $\bg\in L^2(\Omega)$ and $\tau\in\triH$, \label{p:Tdef}
\begin{equation}\label{e:definitionT}
  \int_\tau\A\bgrad(T\mu)\cdot\bgrad\tilde v\,d\xx
  =(\mu,\tilde v)_{\partial\tau},
  \qquad
  \int_\tau\A\bgrad(\tilde T\bg)\cdot\bgrad\tilde v\,d\xx
  =(\bg,\tilde v)_\tau\quad\text{for all } \tilde v\in \tildeVV.
\end{equation}
It follows from the above definition that $\tilde T\bg=0$ if $\bg\in\VVRM$, and that, for all $\mu\in\LH$, 
\begin{equation}\label{e:Ttildecommute}
  (\mu,\tilde T\bg)_{\btriH}
  =\sum_{\tau\in\triH}\int_\tau\A\bgrad(T\mu)\cdot\bgrad(\tilde T\bg)\,d\xx
  =(\bg,T\mu)_{\triH}. 
\end{equation}

Note from~\eqref{e:weak-e3} that $\tilde u=T\lambda+\tilde T\bg$, and substituting in~\eqref{e:weak-e1}, we have that $\uurm\in\VVRM$ and $\lambda\in\LH$ solve 
\begin{equation}\label{e:mc}
\begin{alignedat}{3}
  &(\mu,T\lambda)_{\btriH}+(\mu,\uurm)_{\btriH}&&=-(\mu,\tilde T\bg)_{\btriH}
  \quad&&\text{for all }\mu\in\LH, 
  \\
  &(\lambda,\vvrm)_{\btriH}&&=-(\bg,\vvrm)_{\triH}
  \quad&&\text{for all }\vvrm\in\VVRM.
\end{alignedat}
\end{equation}
We use these unknowns $\uurm$ and $\lambda$ to reconstruct $u$ and the flux $\bsigma$ as follows:
\begin{equation}\label{e:disp-stress}
u=\uurm+\tilde u=\uurm+T\lambda+\tilde T\bg,
\qquad
\bsigma=\A\bgrad(T\lambda+\tilde T\bg).
\end{equation}

\begin{rem}
With the above definitions, it is possible to rewrite the energy norm as below:
\[
|T\lambda|_{H_\A^1(\triH)}^2
=\sum_{\tau\in\triH}(\lambda,T\lambda)_{\partial\tau}
=\|\A^{-1/2}\bsigma_\lambda\|_{L^2(\Omega)}^2,
\quad\text{where }\bsigma_\lambda=\A\bgrad T\lambda. 
\]
\end{rem}
\section{Hybrid Localized Finite Elements}\label{s:Hlocal}
In this section we decompose the space $\LH$ into a direct sum of subspaces, and then rewrite problem~\eqref{e:mc} into an equivalent form that is convenient for our purposes. Consider $\faceh$ be a partition of the faces of elements in $\triH$, refining them in the sense that every (coarse) face of the elements in $\triH$ can be written as a union of faces of $\faceh$. Let $\Lh\subset\Lambda(\triH)$ be the space of piecewise constants on $\faceh$, i.e.,
\begin{equation}\label{e:lambdaconst}
\Lh
=\biggl\{\prod_{\tau\in\triH}\btau\cdot\nn^\tau|_{\dtau}:\,\btau\in\Hdiv,\,\btau\cdot\nn^\tau|_{F_h}\text{ is constant on each face } F_h \in \faceh\biggr\}
\subsetneq\LH.
\end{equation}

There is a delicate aspect regarding the space $\Lambda(\triH)$, since  their elements are \emph{functionals}. In particular, an element   $\mu_h\in\Lh\subset\Lambda(\triH)$ have ``different values,'' with opposite signs at each side of a face $F_h\in\faceh$, if $F_h$ is not a subset of $\dO$. However, to simplify notation, we write $\mu_h|_{F_h}$ to indicate the restriction of $\mu_h$ to $F_h$.

To simplify the presentation we do not discretize $H^1(\tau)$ and $\Hdivtau$ for $\tau\in\triH$. We remark that the method developed here extends easily when we discretize $\Hdivtau$ by simplices or cubic elements with lowest order Raviart--Thomas spaces~\cite{MR1759911}, or discretize $H^1(\tau)$ fine enough to resolve the 
heterogeneities of $\A(x)$ and to satisfy inf-sup conditions with 
respect to the space $\Lh$.  

We pose then the problem of finding $\uurm_h\in\VVRM$ and $\lambda_h\in\Lh$ such that 
\begin{equation}\label{e:weak-hybridh}
\begin{alignedat}{3}
  &(\mu_h,T\lambda_h)_{\btriH}+(\mu_h,\uurm_h)_{\btriH}
  &&=-(\mu_h,\tilde T\bg)_{\btriH}\quad&&\text{for all }\mu_h\in\Lh, 
  \\
  &(\lambda_h,\vvrm)_{\btriH}&&=-(\bg,\vvrm)_{\triH}
  \quad&&\text{for all }\vvrm\in\VVRM. 
\end{alignedat}
\end{equation}
Since~\eqref{e:weak-hybridh} is finite dimensional, it is well-posed if and only if it is injective. Assuming that $\bg=0$, we easily gather that $\lambda_h=0$ and $\uurm_h=0$; see Lemma~\ref{l:lrmsystem}. We define our approximation as in~\eqref{e:disp-stress}, by 
\begin{equation}\label{e:disp-stressh}
u_h=\uurm_h+T\lambda_h+\tilde T\bg,\qquad \bsigma_h=\A\bgrad(T\lambda_h+\tilde T\bg).
\end{equation}
Simple substitutions yield that $u_h$, $\lambda_h$ solve~\eqref{e:weak-hybridH} if $\Lambda$ is replaced by $\Lh$, i.e.,
\begin{equation}\label{e:weak-hybridHh}
\begin{alignedat}{3}
&(\A\bgrad u_h,\bgrad v)_{\triH}-(\lambda_h,v)_{\btriH}&&=(g,v)_{\triH}\quad&&\text{for all }v\in\VH,
\\
&(\mu_h,u_h)_{\btriH}&&=0\quad&&\text{for all }\mu_h\in\Lh.
\end{alignedat}
\end{equation}

We assume that $\Lh$ is chosen fine enough so that 
\begin{equation}\label{e:HHdef}
  |u-u_h|_{H_\A^1(\triH)}^2
  =\bigl(\lambda-\lambda_h,T(\lambda-\lambda_h)\bigr)_{\triH}
 \le \HH^2 \|g\|_{L^2(\Omega)}^2,  
  \end{equation}
where $\HH$ represents a ``target precision'' the method should achieve. For instance, one could choose $\HH=H$ or $h^s$, for $0<s\le1$. 
It must be mentioned that $\lambda_h$ is never computed, the main goal
of this paper is to develop an efficient approximation of order $\HH$
for $\lambda_h$ using $O(H^{-d})$ degrees of freedom. 

Above, and in what follows, $c$ denotes an arbitrary constant that does not depend on $H$, $\HH$, $h$, $\A$, depending only on the shape regularity of the elements of $\triH$. 

Taking a further step, we decompose $\Lh$ into a space of ``constants'' plus ``zero-average'' functionals over the border of the elements of $\triH$. For each $\tau_i\in\triH$, let $\lambdarm_i\in\Lh$ such that 
\begin{equation}\label{e:lrmdef}
  (\lambdarm_i,v)_{\btriH}=\int_{\partial \tau_i}\salto v\,d\xx
  \quad\text{for all }v\in H^1(\triH),
\end{equation}
where $\salto\cdot$ denotes the jump operator, defined as follows. For each face $F$ belonging to the boundaries of two different elements $\tau_i$, $\tau_j$, fix $\nn_F$ as the constant unitary normal vector pointing either inward or outward. If $\nn_F$ is oriented from $\tau_i$ to $\tau_j$, let $\salto v=v_i|_F-v_j|_F$, where $v_i=v|_{\tau_i}$, if not $\salto v=v_j|_F-v_i|_F$. As usual, if $F$ belongs to $\partial\Omega$, then $\nn_F=\nn$ points outward and $\salto v=v$. 
\begin{rem}
It is also possible to define $\lambdarm_i$ explicitly, given $\tau_i\in\triH$. Consider the face $F\in\F_H$. If $F$ does not belong to $\dtau_i$ then $\lambdarm_i|_F=0$. If it does, $\lambdarm_i|_{F\cap\dtau_i}=-\lambdarm_i|_{F\cap\dtau_i}=1$ or $-1$ depending whether $\nn_F$ points outward or inward of $\tau_i$, respectively. Note that $\lambdarm_i\in\Lambda_h$.
\end{rem}

Let $N$ be the number of elements of $\triH$ and 
\begin{equation}\label{e:lammbdadecomp}
\begin{gathered}
\LRM=\spann\{\lambdarm_i:\,i=1,\dots,N\},
\\
\tildeLambda_h
=\VVRM^\perp
=\{\mu_h\in\Lh:\,(\mu_h,\vvrm)_{\btriH}=0\text{ for all }\vvrm\in\VVRM\}. 
\end{gathered}
\end{equation}
We can now decompose $\Lh=\LRM\oplus \tilde\Lambda_h$ as follows~\cite{MR2759829}. Given $\mu_h\in\Lh$, let $\murm\in\LRM$ and $\tilde\mu_h\in\tildeLambda_h$ such that 
\begin{equation*}
(\murm,\vvrm)_{\btriH}=(\mu_h,\vvrm)_{\btriH}, \qquad(\tilde\mu_h,v)_{\btriH}=(\mu_h,v)_{\btriH}-(\murm,v)_{\btriH}, 
\end{equation*}
for all $\vvrm\in\VVRM$ and $v\in \VH$. Note that $\tilde\mu_h\in\tildeLambda_h$ since $(\tilde\mu_h,\vvrm)_{\btriH}=(\mu_h,\vvrm)_{\btriH}-(\murm,\vvrm)_{\btriH}=0$, and $\mu_h = \murm + \tilde\mu_h $. 

We also decompose $\tildeLambda_h=\LTRM\oplus\tildeLambda_h^f$.\label{p:Lfdef}
Basically, the elements of $\LTRM$ are constants on each element face of $\trih$ but still with zero average over the element boundaries, and the elements of $\tildeLambda_h^f$ have zero average on each face: 
\begin{equation}\label{e:lambdafaces}
\begin{gathered}
  \LTRM=\{\mu_h\in\tildeLambda_h:\,\mu_h|_F\text{ is constant for each face }F\subset\partial \tau,\,\tau\in\triH\},
  \\
  \tildeLambda_h^f=\{\mu_h\in\tildeLambda_h:\,\int_F\mu_h\,ds=0\text{ for each face }F\subset\partial\tau,\,\tau\in\triH\}. 
\end{gathered}
\end{equation}

Considering again~\eqref{e:weak-hybridh}, from the decomposition for $\Lh$, we gather that $\lambda_h=\lambdarm+\tlambdarm_h+\tilde\lambda^f_h$.\label{p:ldecomp}
Thus, $\uurm_h\in\VVRM$, $\lambdarm\in\LRM$, $\tlambdarm_h\in\LTRM$ and $\tilde\lambda_h\in\tilde\Lambda_h^f$ solve
\begin{equation}\label{e:mc3}
\begin{alignedat}{3}
  &(\lambdarm,\vvrm)_{\btriH}&&
  =-(\bg,\vvrm)_{\triH}\quad&&\text{for all }\vvrm\in\VVRM,
  \\
  &(\tilde\mu_h^f,T\lambdarm+T\tlambdarm_h+T\tilde\lambda^f_h)_{\btriH}&&
  =-(\tilde\mu_h^f,\tilde T\bg)_{\btriH}\quad&&\text{for all }\tilde\mu^f_h\in\tilde\Lambda_h^f, 
  \\
  &(\tmurm,T\lambdarm+T\tlambdarm_h+T\tilde\lambda_h^f)_{\btriH}&&
  =-(\tmurm,\tilde T\bg)_{\btriH}\quad&&\text{for all }\tmurm\in\LTRM, 
  \\
  &(\murm,T\lambdarm+T\tlambdarm_h+T\tilde\lambda_h^f)_{\btriH}+(\murm,\uurm_h)_{\btriH}&&
  =-(\murm,\tilde T\bg)_{\btriH}\quad&&\text{for all }\murm\in\LRM. 
\end{alignedat}
\end{equation}

It is possible to compute the unknowns step-by-step as we detail below. After that we discuss the well-posedness of each problem. The first equation of~\eqref{e:mc3} determines $\lambdarm$. To deal with the second equation, we define the operator $P:\VH\to\tilde\Lambda_h^f$ such that, for $w\in\VH$, 
\begin{equation}\label{e:Pdef}
(\tilde\mu^f_h,TPw)_{\btriH}
=(\tilde\mu^f_h,w)_{\btriH}
\quad\text{for all }\tilde\mu^f_h\in\tilde\Lambda^f_h, 
\end{equation}
i.e., $\bigl(\tilde\mu_h^f,(I-TP)w\bigr)_{\btriH}=0$. Note that $PT$ is an orthogonal projection from $\Lambda_h$ to $\tilde\Lambda^f_h$ since 
\begin{equation*}
(\tilde\mu^f_h,TPT\lambda_h)_{\btriH}
=(\tilde\mu_h^f,T\lambda_h)_{\btriH}
\quad\text{for all }\tilde\mu_h^f\in\tilde\Lambda^f_h.
\end{equation*}

The second equation of~\eqref{e:mc3} becomes 
\begin{equation}\label{e:lambdadef}
\tilde\lambda_h^f=-P(T\lambdarm+T\tlambdarm_h+\tilde T\bg).
\end{equation}
Solving~\eqref{e:Pdef} efficiently is crucial for the good performance of the method, since it is the only large dimensional system of~\eqref{e:mc3}, in the sense that its size grows with order of $h^{-d}$.
This issue is treated in Section~\ref{ss:lowcon} by taking into
account the exponential decay of $PT(\lambdarm+\tlambdarm_h)$. It is
also required to compute or to approximate $\tilde T\bg$ and $P\tilde T\bg$
efficiently. These issues are treated in Sections~\ref{ss:lowcon} and~\ref{ss:findim}. 

Now, we can write the third equation of~\eqref{e:mc3} as
\begin{equation}\label{e:lrmaux}
(\tmurm,T\tlambdarm_h)_{\btriH}
=-(\tmurm,\tilde T\bg)_{\btriH}-\bigl(\tmurm,T\lambdarm-TP(T\lambdarm+T\tlambdarm_h+\tilde T\bg)\bigr)_{\btriH}
\end{equation}
and then
\begin{equation*}
(\tmurm,T\tlambdarm_h-TPT\tlambdarm_h)_{\btriH}
=-(\tmurm,\tilde T\bg-TP\tilde T\bg)_{\btriH}-(\tmurm,T\lambdarm-TPT\lambdarm)_{\btriH}. 
\end{equation*}
Since $PT\tmurm\in\tilde\Lambda_h^f$, 
\begin{multline*}
  \bigl(\tmurm-PT\tmurm,(I-TP)T\tlambdarm_h\bigr)_{\btriH}
  \\
  =-(\tmurm-PT\tmurm,(I-TP)\tilde T\bg)_{\btriH}-(\tmurm-PT\tmurm,(I-TP)T\lambdarm)_{\btriH}. 
\end{multline*}
Thus, $\tlambdarm_h$ is computed from 
\begin{multline}\label{e:tlrmeqtn}
\bigl((I-PT)\tmurm,T(I-PT)\tlambdarm_h\bigr)_{\btriH}
=-\bigl((I-PT)\tmurm,(I-TP)\tilde T\bg\bigr)_{\btriH}
\\
-\bigl((I-PT)\tmurm,T(I-PT)\lambdarm\bigr)_{\btriH}
\quad\text{for all }\tmurm\in\LTRM, 
\end{multline}
and $\tilde\lambda^f_h$ is recovered from~\eqref{e:lambdadef}.

We note that $((I-PT)\tmurm,TP\tilde T\bg\bigr)_{\btriH} = 0$ since
\[
(PT\tmurm,TP\tilde T\bg\bigr)_{\btriH} = (TPT\tmurm,P\tilde T\bg\bigr)_{\btriH}
= (T\tmurm,P\tilde T\bg\bigr)_{\btriH} = (\tmurm,TP\tilde T\bg\bigr)_{\btriH}
\]
where we have used the symmetry of $T$, the definition of $P$ and again
the symmetry of $T$. This simplification in ~\eqref{e:tlrmeqtn} is very
important since $P\tilde T\bg $  cannot be pre-processed. Finally, the fourth equation of~\eqref{e:mc3} yields $\uurm_h$, and the post-processing~\eqref{e:disp-stressh} recovers the main variables:
\begin{equation}\label{disp-stress2}
  u_h=\uurm_h+T(\lambdarm+\tlambdarm_h+\tilde\lambda^f_h)+\tilde T\bg,
  \qquad
  \bsigma_h=\A\bgrad [T(\lambdarm+\tlambdarm_h+\tilde\lambda^f_h)+\tilde T\bg], 
\end{equation}
yielding $u_h$ and $\lambda_h=\lambdarm+\tlambdarm_h+\tilde\lambda^f_h$, solutions of~\eqref{e:weak-hybridh}.

To show the existence and uniqueness of solutions for~\eqref{e:mc3}, we proceed by parts. The existence of solution for the first equation follows from Lemma~\ref{l:lrmsystem}. Solving the second equation is equivalent to~\eqref{e:Pdef}, and such system is well-posed due to the coercivity of $(\cdot,T\cdot)_{\btriH}$ on $\tilde\Lambda_h^f$; see~\cites{AHPV,HMV} and~\cites{MR1802366,MR1921914,MR2104179}. The same arguments hold for the third equation of~\eqref{e:mc3}, rewritten in~\eqref{e:lrmaux}. Another way to see this is to consider~\eqref{e:tlrmeqtn} with zero right hand side. From the coercivity of $(\cdot,T\cdot)_{\btriH}$ on $\tildeLambda$ we have $(I-PT)\tlambdarm_h=0$. But since $\LTRM\cap\tilde\Lambda_h^f=\{0\}$, then $\tlambdarm_h=0$. Finally, the fourth equation of~\eqref{e:mc3} is again finite dimension, and if $(\murm,\uurm_h)_{\btriH}=0$ for all $\murm\in\LRM$, then, from Lemma~\ref{l:lrmsystem}, $\uurm_h=0$.

Henceforth, we consider the following algorithm, based on~\eqref{e:mc3}:
\begin{enumerate}[(i)]
\item{\bf Find $\lambdarm\in\LRM$ from the first equation of~\eqref{e:mc3}}. 
\item{\bf Find $\tlambda_h^0$ from~\eqref{e:tlrmeqtn}.} 
\item{\bf Find $\tilde\lambda_h^f$ from~\eqref{e:lambdadef}.}
\item{\bf Find $u_h^0$ from from the fourth equation of~\eqref{e:mc3}.}
\item{\bf Compute $u_h$ and $\bsigma_h$ from~\eqref{disp-stress2}.} 
\end{enumerate}
Except for (ii), all steps above above can be performed efficiently as the matrices involved are sparse and either local or independent of $h$. Solving~\eqref{e:tlrmeqtn} on the other hand involves computing the  $h$-dependent, global operator $P$, leading to a dense matrix in~\eqref{e:tlrmeqtn}. From now on, we concentrate on approximating $P$ so that~\eqref{e:tlrmeqtn} can be accurately and efficiently solved.

The key to approximate~\eqref{e:tlrmeqtn} is the exponential decay of $Pw$, as long as $w\in\VH$ has local support. That allows replacing $P$ by a semi-local operator $P^j$. That works fine for low-contrast coefficients and is the subject of Section~\ref{ss:lowcon}. For high-contrast coefficients however, the exponential decay rate is smaller, and to circumvent that we consider in Section~\ref{ss:dhc} a spectral decomposition of $\tilde\Lambda_h^f$. 

\subsection{Decaying High-Contrast}\label{ss:dhc}
It is essential for the  performing method that the static condensation is done efficiently. The solutions of~\eqref{e:Pdef} decay exponentially fast if $w$ has local support, so instead of solving the problems in the whole domain it would be reasonable to solve it locally using patches of elements. We note that the idea of performing global static condensation goes back to the Variational Multiscale Finite Element Method--VMS~\cites{MR1660141,MR2300286}. Recently variations of the VMS
and denoted by Localized Orthogonal Decomposition Methods--LOD were introduced and analyzed in~\cites{MR3246801,MR2831590,MR3552482,MR3591945}. 

The main bottle-neck in dealing with high-contrast coefficients is that the decay is slower, albeit still exponential, forcing the use of larger patches. To deal with this situation, we use a subspace of $\tilde\Lambda_h^f$ to augment $\LTRM$ by selecting eigenfunctions associated to a proper generalized eigenvalue problem associated to each face of the mesh $\triH$. We then remove the ``slow decaying modes''from $\tilde\Lambda_h^f$. To define these  generalized eigenvalue problems, we first introduce some theoretical tools for high-contrast coefficients.


For $K\in\triH$, define $\T_0(K)=\emptyset$, $\T_1(K)=\{K\}$, and for $j=1,2,\dots$ let 
\begin{equation}\label{e:taujk}
\T_{j+1}(K)=\{\tau\in\triH:\,\overline\tau\cap\overline\tau_j\ne\emptyset\text{ for some }\tau_j\in\T_j(K)\}. 
\end{equation}

Let $\tau\in\triH$, $F$ a face of $\partial\tau$, and let $F_\tau^c=\partial\tau\backslash F$. \label{p:Fcdef}
Define
\begin{equation}\label{e:hcspaces}
\tilde\Lambda_h^\tau
=\{\tilde\mu_h|_{\partial\tau}:\,\tilde\mu_h\in\tildeLambda^f_h\},
\qquad
\tilde\Lambda_h^F
=\{\tilde\mu_h|_{F}:\,\tilde\mu_h\in\tilde\Lambda_h^f\},
\qquad
\tilde\Lambda_h^{F_\tau^c}
=\{\tilde\mu_h|_{F_\tau^c}:\,\tilde\mu_h\in\tilde\Lambda_h^f\}.
\end{equation}
Abusing the notation, we identify $\tilde\Lambda_h^\tau=\tilde\Lambda_h^F\times\tilde\Lambda_h^{F_\tau^c}$, and for any given $\tilde\mu_h^\tau\in\tilde\Lambda_h^\tau$, denote $\tilde\mu_h^\tau=\{\tilde\mu_h^F,\tilde\mu_h^{F_\tau^c}\}\in\tilde\Lambda_h^F\times\tilde\Lambda_h^{F_\tau^c}$. Define
\begin{gather*}
T_{FF}^\tau: \tilde\Lambda_h^F \to(\tilde\Lambda_h^F)^\prime, 
\qquad
T_{F^cF}^\tau: \tilde\Lambda_h^{F} \to(\tilde\Lambda_h^{F_\tau^c})^\prime,
\\
T_{FF^c}^\tau: \tilde\Lambda_h^{F_\tau^c} \to(\tilde\Lambda_h^{F})^\prime,
\qquad
T_{F^cF^c}^\tau: \tilde\Lambda_h^{F_\tau^c} \to(\tilde\Lambda_h^{F_\tau^c})^\prime,
\end{gather*}
by
\begin{multline*}
  (\tilde\mu_h,T \tilde\mu_h)_{\partial\tau}=: \bigl(\{\tilde\mu_h^F,\tilde\mu_h^{F_\tau^c}\},
  T\{\tilde\mu_h^F,\tilde\mu_h^{F_\tau^c}\}\bigr)_{\partial\tau}
  \\
=(\tilde\mu^F_h, T_{FF}^\tau\tilde\mu_h^F)_F +
(\tilde\mu^F_h, T_{FF^c}^\tau\tilde\mu_h^{F_\tau^c})_F +
(\tilde\mu^{F_\tau^c}_h,T_{F^cF}^\tau\tilde\mu_h^{F})_{F_\tau^c} +
(\tilde\mu^{F_\tau^c}_h,T_{F^cF^c}^\tau\tilde\mu_h^{F_\tau^c})_{F_\tau^c}.
\end{multline*}
We remind that $T:\tildeLambda_h^\tau\to\widetilde H^1(\tau)$, satisfies
\begin{equation}\label{p:Ttaudef}
(\mathcal A\bgrad(T\tilde\mu^\tau_h),\bgrad v)_\tau
=(\tilde\mu_h^\tau,v)_{\partial\tau}\quad\text{for all }v\in\widetilde H^1(\tau), 
\end{equation}
and $\A\bgrad(T \tilde\mu^\tau_h)\cdot\nn^\tau=\tilde\mu_h$ on $\partial\tau$.

It follows from the properties of $T$ that both $T_{FF}^\tau$ and $T_{F^cF^c}^\tau$ are symmetric and positive definite matrices, and follow from Schur complement arguments that for any  $\{\tilde\mu^F_h,\tilde\mu^{F_\tau^c}_h\}\in\tilde\Lambda_h^\tau$
\begin{multline}\label{e:inequalityenergy2}
(\tilde\mu^F_h, \hat{T}_{FF}^\tau\tilde\mu_h^F)_F
=\min_{\tilde\nu^{F_\tau^c}_h\in\tilde\Lambda_h^{F_\tau^c}}
(\{\tilde\mu^F_h,\tilde\nu^{F_\tau^c}_h\},T\{\tilde\mu_h^F,\tilde\nu^{F_\tau^c}_h\})_{\partial \tau}
\\
\le \bigl(\{\tilde\mu_h^F,\tilde\mu_h^{F_\tau^c}\},
T\{\tilde\mu_h^F,\tilde\mu_h^{F_\tau^c}\}\bigr)_{\partial\tau}
=(\tilde\mu_h^\tau,T\tilde\mu_h^\tau)_{\partial\tau}, 
\end{multline}
where
\[
\hat T_{FF}^\tau=T_{FF}^\tau-T_{FF^c}^\tau(T_{F^cF^c}^\tau)^{-1}T_{F^cF}^\tau, 
\]
and the minimum is attained at $\tilde\nu_h^{F_\tau^c}=-(T_{F^cF^c}^\tau)^{-1}T_{F^cF}^\tau\tilde\mu^F_h$. In what follows, to take into account high contrast coefficients, we consider the following generalized eigenvalue problem: Find eigenpairs $(\alpha_{i}^F,\tilde{\mu}_{h,i}^F)\in(\RR,\tilde\Lambda_h^F)$, where
$\alpha_{1}^F \le \alpha_{2}^F \le \alpha_{3}^F,\dots$, such that 
\begin{enumerate}[I)]
\item\label{e:eig1}  If the face $F$ is shared by elements $\tau$ and $\tau^\prime$ we solve
\begin{equation}\label{e:eqeig1}
\bigl(
\tilde\nu^F_h,(T_{FF}^\tau+T_{FF}^{\tau^\prime})\tilde\mu_{h,i}^F
\bigr)_F
=\alpha_i^F
\bigl(
\tilde\nu_h^F,(\hat T_{FF}^\tau+\hat T_{FF}^{\tau^\prime})\tilde\mu_{h,i}^F
\bigr)_F
\quad\text{for all }\tilde\nu^F_h\in\tilde\Lambda_h^F. 
\end{equation}
\item\label{e:eig2} If the face $F$ is on the boundary $\dO$ we solve 
\begin{equation}\label{e:eqeig2}
\bigl(\tilde\nu^F_h,T_{FF}^\tau\tilde{\mu}_{h,i}^F\bigr)_F
=\alpha_i^F\bigl(\tilde\nu_h^F,\hat T_{FF}^\tau\tilde\mu_{h,i}^F\bigr)_F,
\quad\text{for all }\tilde\nu_h^F\in\tilde\Lambda_h^F. 
\end{equation}
\end{enumerate}

Let $\astab\ge 1$ be a user-defined constant, aiming to ``control'' the localization of basis functions. Roughly, the basis functions become more localized with smaller $\astab$, in a sense that we make clear further on. 

We  decompose $\tilde\Lambda_h^F:=\tilde\Lambda_h^{F,\triangle}\oplus\tilde\Lambda_h^{F,\Pi}$ as
\begin{equation} \label{e:Lpmdef}
\tilde\Lambda_h^{F,\triangle}:=\spann\{\tilde\mu_{h,i}^F:\,\alpha_{i}^F < \alpha_{\rm{stab}}\},
\qquad
\tilde\Lambda_h^{F,\Pi}:=\spann\{\tilde\mu_{h,i}^F:\,\alpha_{i}^F \ge \alpha_{\rm{stab}}\}.
\end{equation}

The eigenfunctions $\tilde{\mu}_{h,i}^F$ are chosen to be orthonormal with respect to $(\cdot, (\hat{T}_{FF}^\tau + \hat{T}_{FF}^{\tau^\prime})\cdot)_F$ if $F$ is an interior face, and with $(\cdot, \hat{T}_{FF}^\tau \cdot)_F$ if $F\subset\partial\Omega$. 

Generalized eigenvalue problems of this type have appeared in the literature to 
make preconditioners robust with respect to coefficients~\cites{MR2277024,MR2334131,MR3089678,DP2013,MR3350292,MR3303686,MR3546980,MR3612901,MR3678572,MR3582898}. In particular~\cite{ZAMPINI} shows, for a similar problem, that $\alpha_i^F-1$ decays exponentially to zero since, when $h$ goes to zero, the operators $\hat T_{FF}$
(related to $(H^{1/2}(F))^\prime$)  and $T_{FF}$
(related to $(H^{1/2}_{00}(F))^\prime$) differ only by a compact operator. We
note that in~\cite{ZAMPINI} and also here, the generalized eigenvalue problems
have a dual algebraic relation with the parallel sums first introduced
in the classical article by~\cite{MR242573}. Additionally, since there are
no degrees of freedom at the vertices (in 2D) or edges (in 3D), we do not have
generalized eigenvalue problems associated to  vertices (in 2D) or
edges (in 3D), respectively.

In~\cite{MR2718268} is shown that the number of eigenvalues that are very large is related to the number of connected sub-regions on $\bar\tau\cup{\bar\tau}^\prime$ with large coefficients surrounded by regions with small coefficients. Generalized eigenvalue problems also have been used on overlapping domain decomposition solvers~\cites{MR2718268,MR2916377,MR3175183,MR3033238}. The design of robust discretizations with respect to coefficients using domain decomposition ideas have been studied in~\cites{MR2666649,MR1642758,MR3350765} assuming some regularity on the solution, and in~\cite{MR2718268} for a class of problems when the weighted Poincar\'e constant~\cites{MR3047947,MR3013465,MR2867661} is not large, otherwise the exponential decay of the multiscale functions deteriorates. See also~\cites{MR2753343,MR3109775} where a priori error estimates are obtained in terms of spectral norms. 

In the following result, we provide stability estimates for $T$ acting on certain subspaces of $\tilde\Lambda_h^f$, and some upper bounds for the generalized eigenvalues. 

\begin{lemma}\label{l:eigenbound}
Let $F$ be a face of $\dtau$ shared by the elements $\tau$, $\tau^\prime\in\triH$. If $F$ is on the boundary $\dO$, assume $\tau^\prime=\emptyset$. Then
\begin{equation}\label{e:traceestdelta}
|T\{\tilde\mu_h^{F,\triangle},0\}|_{H_\A^1(\tau)}^2
+|T\{\tilde\mu_h^{F,\triangle},0\}|_{H_\A^1(\tau^\prime)}^2
\le\astab\bigl(
|T\tilde\mu_h^\tau|_{H_\A^1(\tau)}^2+|T\tilde\mu_h^{\tau^\prime}|_{H_\A^1(\tau^\prime)}^2
\bigr)
\end{equation}
for all $\tilde\mu_h^\tau=\{\tilde\mu_h^{F,\triangle},\tilde\mu_h^{F_\tau^c}\}\in\tilde\Lambda_h^{F,\triangle}\times\tilde\Lambda_h^{F_\tau^c}$ and all $\tilde\mu_h^{\tau^\prime}=\{\tilde\mu_h^{F,\triangle},\tilde\mu_h^{F_{\tau^\prime}^c}\}\in\tilde\Lambda_h^{F,\triangle}\times\tilde\Lambda_h^{F_{\tau^\prime}^c}$. 
Also, 
\begin{equation}\label{e:traceest}
|T\{\tilde\mu_h^F,0\}|_{H_\A^1(\tau)}^2\le\alpha|T\tilde\mu_h^\tau|_{H_\A^1(\tau)}^2
\quad\text{for all }\tilde\mu_h^\tau
=\{\tilde\mu_h^F,\tilde\mu_h^{F_\tau^c}\}\in\tilde\Lambda_h^F\times\tilde\Lambda_h^{F_\tau^c}, 
\end{equation}
where $\alpha=\gamma\kappa\beta_{H/h}^2$,  the positive constant $\gamma$ depends only on the shape regularity of $\triH$, and 
\[
\beta_{H/h}=1+\log(H/h),
\quad\kappa=\max_{\tau\in\triH}\kappa^\tau,
\quad \kappa^\tau=\frac{\amax^\tau}{\amin^\tau},
\quad \amax^\tau=\sup_{\xx\in\tau}a_+(\xx),
\quad\amin^\tau=\inf_{\xx\in\tau}a_-(\xx)\label{p:betadef}.
\]
Finally, if $\alpha_i^F$ is an eigenvalue of~{(\ref{e:eig1}--\ref{e:eig2})}, then $1\le\alpha_i^F\le\alpha$. 
\end{lemma}
\begin{proof} 
Assume that $F$ is a face shared by $\tau$, $\tau^\prime\in\triH$. The case when $F$ is on the boundary $\dO$ is similar. To show~\eqref{e:traceestdelta}, let $\tilde\mu_h^{F,\triangle}\in\tilde\Lambda_h^{F,\triangle}$. Then the eigen-expansion  $\tilde\mu_h^{F,\triangle}=\sum_ic_i\tilde\mu_{h,i}^F$ holds and $\alpha_j\le\astab$. Then 
\begin{multline*}
|T\{\tilde\mu_h^{F,\triangle},0\}|_{H_\A^1(\tau)}^2
+|T\{\tilde\mu_h^{F,\triangle},0\}|_{H_\A^1(\tau^\prime)}^2
=(\tilde\mu_h^{F,\triangle},T_{FF}^\tau\tilde\mu_h^{F,\triangle})_F
+(\tilde\mu_h^{F,\triangle},T_{FF}^{\tau^\prime}\tilde\mu_h^{F,\triangle})_F
\\
=\sum_{i,j}c_ic_j\bigl(
(\tilde\mu_{h,i}^{F,\triangle},T_{FF}^\tau\tilde\mu_{h,j}^{F,\triangle})_F
+(\tilde\mu_{h,i}^{F,\triangle},T_{FF}^{\tau^\prime}\tilde\mu_{h,j}^{F,\triangle})_F
\bigr)
\\
=\sum_{i,j}c_ic_j\alpha_j\bigl(
(\tilde\mu_{h,i}^{F,\triangle},\hat T_{FF}^\tau\tilde\mu_{h,j}^{F,\triangle})_F
+(\tilde\mu_{h,i}^{F,\triangle},\hat T_{FF}^{\tau^\prime}\tilde\mu_{h,j}^{F,\triangle})_F
\bigr)
\\
=\sum_{j}c_j^2\alpha_j\bigl(
(\tilde\mu_{h,j}^{F,\triangle},\hat T_{FF}^\tau\tilde\mu_{h,j}^{F,\triangle})_F
+(\tilde\mu_{h,j}^{F,\triangle},\hat T_{FF}^{\tau^\prime}\tilde\mu_{h,j}^{F,\triangle})_F
\bigr)
\\
\le\astab\bigl(
(\tilde\mu_h^{F,\triangle},\hat T_{FF}^\tau\tilde\mu_h^{F,\triangle})_F
+(\tilde\mu_h^{F,\triangle},\hat T_{FF}^{\tau^\prime}\tilde\mu_h^{F,\triangle})_F
\bigr)
\le\astab\bigl(
(\tilde\mu_h^\tau,T\tilde\mu_h^\tau)_\dtau+(\tilde\mu_h^{\tau^\prime},T\tilde\mu_h^{\tau^\prime})_{\dtau^\prime} 
\bigr).
\end{multline*}
To show~\eqref{e:traceest}, let us first define $T_\I$ by~\eqref{e:definitionT} with $\A=\I$, the identity tensor, that is, $T_\I$ is the classical harmonic extension with Neumann boundary condition. From Lemma~\ref{l:extensions} in the Appendix and a direct application of~\cite{MR1759911}*{Lemma 4.4} since both $\tilde\mu_F$ and $\tilde\mu_h^\tau$ have zero average on $\partial\tau$, we have  
\begin{multline*}
|T\{\tilde\mu_h^F,0\}|^2_{H_\A^1(\tau)}
\le\frac1{\amin^\tau}|T_\I\{\tilde\mu_h^F,0\}|_{H^1(\tau)}^2
\le\frac{c_1}{\amin^\tau}|\{\tilde\mu_h^F,0\}|^2_{H^{-1/2}(\dtau)}
\le\frac{c_2}{\amin^\tau}\beta_{H/h}^2 |\tilde\mu_h^\tau|_{H^{-1/2}(\dtau)}^2
\\
\le\frac\gamma{\amin^\tau}\beta_{H/h}^2|T_\I\tilde\mu_h^\tau|^2_{H^1(\tau)}
\le\gamma\frac{\amax^\tau}{\amin^\tau} \beta_{H/h}^2 |T\tilde\mu_h^\tau|^2_{H_\A^1(\tau)},
\end{multline*}
and then~\eqref{e:traceest} follows.

Next, taking $\tilde\nu_h^{F_\tau^c}=\tilde\mu_h^{F_\tau^c}$ in~{(\ref{e:eig1}--\ref{e:eig2})}, we gather from~\eqref{e:inequalityenergy2} that $\alpha_{i}^F\ge1$. Finally, without loss of generality assume that~\eqref{e:eig1} holds and $\tau$ and $\tau^\prime$ share the face $F$. Then
\begin{multline*}
\alpha_i^F\bigl(\tilde\mu_{h,i}^F,(\hat T_{FF}^\tau+\hat T_{FF}^{\tau^\prime})\tilde\mu_{h,i}^F\bigr)_F
=\bigl(\tilde\mu_{h,i}^F,(T_{FF}^\tau+T_{FF}^{\tau^\prime})\tilde\mu_{h,i}^F\bigr)_F
\\
=(\{\tilde\mu_{h,i}^F,0\},T\{\tilde\mu_{h,i}^F,0\})_\dtau
+(\{\tilde\mu_{h,i}^F,0\},T\{\tilde\mu_{h,i}^F,0\})_{\dtau^\prime}
\\
\le\alpha\bigl(
\min_{\tilde\nu^{F_\tau^c}_h\in\tilde\Lambda_h^{F_\tau^c}}
(\{\tilde\mu_{h,i}^F,\tilde\nu_h^{F_\tau^c}\},T\{\tilde\mu_{h,i}^F,\tilde\nu_h^{F_\tau^c}\})_\dtau
+\min_{\tilde\nu^{F_{\tau^\prime}^c}_h\in\tilde\Lambda_h^{F_{\tau^\prime}^c}}
(\{\tilde\mu_{h,i}^F,\tilde\nu_h^{F_{\tau^\prime}^c}\},T\{\tilde\mu_{h,i}^F,\nu_h^{F_{\tau^\prime}^c}\})_{\dtau^\prime}
\bigr)
\\
=\alpha\bigl(
\tilde\mu_{h,i}^F,(\hat T_{FF}^\tau+\hat T_{FF}^{\tau^\prime})\tilde\mu_{h,i}^F
\bigr)_F, 
\end{multline*}
where we used~\eqref{e:inequalityenergy2} and~\eqref{e:traceest}. Thus, $\alpha_i^F\le\alpha$. 
\end{proof}

To define our LSD--Localized Spectral Decomposition Method for high-contrast coefficients, let us introduce the non-localized version. Let us first define
\begin{equation}\label{e:ltpmdef}
\begin{gathered}
\tilde\Lambda_h^\Pi
=\{\tilde\mu_h\in\tildeLambda_h^f:\,\tilde\mu_h|_F\in\tilde\Lambda_h^{F,\Pi}\text{ for all } F\in \btriH\},
\\ 
\tilde\Lambda_h^\triangle
=\{\tilde\mu_h\in\tildeLambda_h^f:\,\tilde\mu_h|_F\in\tilde\Lambda_h^{F,\triangle}\text{ for all } F\in \btriH\}. 
\end{gathered}
\end{equation}
From Lemma~\ref{l:eigenbound}, if $\astab>\alpha$ then $\tilde\Lambda_h^\Pi$ is empty. 

We now follow the same procedure as in~\eqref{e:mc3} except that now we replace $\LTRM$ by
\begin{equation}\label{e:Ltopidef}
\widetilde\Lambda_h^{0,\Pi}
:=\LTRM\oplus\widetilde\Lambda_h^\Pi,
\end{equation}
replace $\widetilde\Lambda_h^f$ by $\widetilde\Lambda_h^\triangle$ and replace $P$ by $P^\triangle:\VH\to\widetilde\Lambda^\triangle_h$ such that, for $w\in\VH$, 
\begin{equation}\label{e:Pdefhigh}
  (\tilde\mu_h^\triangle,TP^\triangle w)_\btriH=(\tilde\mu_h^\triangle,w)_\btriH\quad\text{for all }\tilde\mu_h^\triangle\in\tilde\Lambda_h^\triangle.
\end{equation}
The outcome is that ${\tilde\lambda^{0,\Pi}}_h\in\widetilde\Lambda_h^{0,\Pi}$ solves 
(cf. with~\eqref{e:tlrmeqtn})
\begin{multline*}
\bigl((I-P^\triangle T){\tilde\mu^{0,\Pi}}_h,T(I-P^\triangle T){\tilde\lambda^{0,\Pi}}_h\bigr)_{\btriH}
=-\bigl((I-P^\triangle T){\tilde\mu^{0,\Pi}}_h,(I-TP^\triangle)\tilde T\bg\bigr)_{\btriH}
\\
-\bigl((I-P^\triangle T){\tilde\mu^{0,\Pi}}_h,T(I-P^\triangle T)\lambdarm\bigr)_{\btriH}
\quad\text{for all }{\tilde\mu^{0,\Pi}}_h\in\widetilde\Lambda_h^{0,\Pi}, 
\end{multline*}
and 
\begin{equation}\label{e:alternative}
\begin{gathered}
u_h=u_h^0+T\lambda_h+\tilde T\bg, 
\qquad
\tilde\lambda_h^\triangle=-P^\triangle(T\lambdarm+T{\tilde\lambda^{0,\Pi}}_h+\tilde T\bg),
\qquad
\\
\lambda_h
=\lambdarm+\tilde\lambda_h^{0,\Pi}+\tilde\lambda_h^\triangle
=(I-P^\triangle T)(\lambdarm+\tilde\lambda_h^{0,\Pi})-P^\triangle\tilde T\bg. 
\end{gathered}
\end{equation}
Note that $u_h$, $u_h^0$ and $\lambdarm  $ in~\eqref{e:alternative} are the same as in~\eqref{disp-stress2}. Also, since the space $\tildeLambda_h=\LTRM\oplus\tilde\Lambda_h^\Pi \oplus \tilde\Lambda_h^\triangle$ is a direct sum,
\begin{equation}\label{e:decompPitriang}
\tilde\lambda_h^{0,\Pi}=\tilde\lambda_h^{0}+\tilde\lambda_h^\Pi,
\qquad
    \tilde\lambda_h^\Pi+\tilde\lambda_h^\triangle=\tilde\lambda_h^f,
\end{equation}
where $\tilde{\lambda}_h^{0}$ and $\tilde\lambda_h^f$ are the same as in~\eqref{disp-stress2} and $\tilde{\lambda}_h^{\Pi}\in\tilde\Lambda_h^\Pi$.  

The lemma that follows provides a crucial estimate for solutions of~\eqref{e:Pdefhigh}, allowing the establishment of the exponential decay of solutions in Theorem~\ref{t:decaythmhc}. 

\begin{lemma}\label{l:decayhc}
Let $v\in\VH$ such that $\supp v\subset K$, and $\tilde\mu_h^{\triangle}=P^{\triangle}v$. Then, for any integer $j\ge1$, 
\begin{equation*}
  |T\tilde\mu_h^{\triangle}|_{H_\A^1(\triH\backslash {\mathcal{T}}_{j+1}(K))}^2
  \le \Nf\alpha_{\rm{stab}} |T\tilde\mu_h^{\triangle}|_{H_\A^1(\T_{j+1}(K)\backslash\T_{j-1}(K))}^2, 
\end{equation*}
where $\Nf$ is the maximum number of faces that an element might have. 
\end{lemma}
\begin{proof} 
Choose $\tilde\nu_h^\triangle\in\tilde\Lambda_h^\triangle$ defined by $\tilde\nu_h^\triangle|_F=0$ if $F$ is a face of an element of $\T_j(K)$ and $\tilde\nu_h^\triangle|_F=\tilde\mu_h^\triangle|_F$ otherwise. We obtain 
\begin{multline*}
|T\tilde\mu_h^\triangle|_{H_\A^1(\triH\backslash{\T}_{j+1}(K))}^2
=\sum_{\tau\in\triH\backslash\T_{j+1}(K)}(\tilde\mu_h^\triangle,T\tilde\mu_h^\triangle)_{\partial\tau}
\\
=\sum_{\tau\in\triH}(\tilde\nu_h^\triangle,T\tilde\mu_h^\triangle)_{\partial\tau}
-\sum_{\tau\in\T_{j+1}(K)\backslash\T_j(K)}(\tilde\mu_h^\triangle,T\tilde\mu_h^\triangle)_{\partial\tau}
+\sum_{\tau\in\T_{j+1}(K)\backslash\T_{j}(K)}(\tilde\mu_h^\triangle-\tilde\nu_h^\triangle,T\tilde\mu_h^\triangle)_{\partial\tau}
\\
=-|T\tilde\mu_h^\triangle|_{H_\A^1({\T}_{j+1}(K)\backslash{\T}_{j}(K))}^2
+\sum_{\tau\in\T_{j+1}(K)\backslash\T_{j}(K)}
(\tilde\mu_h^\triangle-\tilde\nu_h^\triangle,T\tilde\mu_h^\triangle)_{\partial\tau}, 
\end{multline*}
where we used that $\sum_{\tau\in\triH}(\tilde\nu_h^\triangle,T\tilde\mu_h^\triangle)_{\partial\tau}=0$ due to the definition of $T$ and the local support of $v$.

For each $\tau\in\T_{j+1}(K)\backslash\T_j(K)$ and let $\hat\F(\tau)$ be the set of faces of $\tau$ that are also faces of an element in $\T_{j}(K)$. Let $\tau^\prime(F)\in\T_j(K)$ be the element sharing $F$. 
\begin{multline*}
\sum_{\tau\in\T_{j+1}(K)\backslash\T_{j}(K)}
(\tilde\mu_h^\triangle-\tilde\nu_h^\triangle,T\tilde\mu_h^\triangle)_{\partial\tau}
=\sum_{\tau\in\T_{j+1}(K)\backslash\T_{j}(K)}
\sum_{\hat F\in\hat\F(\tau)}(\tilde\mu_h^\triangle,T\tilde\mu_h^\triangle)_{\hat F}
\\
\le\astab\sum_{\tau\in\T_{j+1}(K)\backslash\T_{j}(K)}
\sum_{\hat F\in\hat\F(\tau)}\bigl(
|T\tilde\mu_h^\triangle|_{H_\A^1(\tau)}^2+|T\tilde\mu_h^\triangle|_{H_\A^1(\tau^\prime(F))}^2
\bigr)
\\
\le \Nf\astab\bigl(
|T\tilde\mu_h^\triangle|_{H_\A^1(\T_{j+1}(K))}^2+|T\tilde\mu_h^\triangle|_{H_\A^1(\T_j(K))}^2
\bigr)
\end{multline*}
from~\eqref{e:traceestdelta}, and since $\hat\F(\tau)$ contains at most $\Nf$ faces. 
\end{proof} 
\begin{rem}
The conclusion of Lemma~\ref{l:decayhc} also holds if $\supp v\subset\T_i(K)$ for some positive integer $i<j$.
\end{rem}

\begin{theorem}\label{t:decaythmhc}
Let $v\in\VH$ such that $\supp v\subset K$, and $\tilde\mu_h^\triangle=P^\triangle v$. Then, for any integer $j\ge1$, 
\begin{equation*}
|T\tilde\mu_h^\triangle|_{H_\A^1(\triH\backslash\T_{j+1}(K))}^2
\le e^{-\frac{[(j+1)/2]}{1+\Nf\astab}}|T\tilde\mu_h^\triangle|_{H_\A^1(\triH)}^2,
\end{equation*}
where $[s]$ is the integer part of $s$, and $\Nf$ is the maximum number of faces that an element might have.
\end{theorem}
\begin{proof} 
Using Lemma~\ref{l:decayhc} we have 
\begin{equation*}
|T\tilde\mu_h^\triangle|_{H_\A^1(\triH\backslash\T_{j+1}(K))}^2
\le \Nf\astab|T\tilde\mu_h^\triangle|_{H_\A^1(\triH\backslash\T_{j-1}(K))}^2
-\Nf\astab|T\tilde\mu_h^\triangle|_{H_\A^1(\triH\backslash{\T}_{j+1}(K))}^2, 
\end{equation*}
and then
\[
|T\tilde\mu_h^\triangle|_{H_\A^1(\triH\backslash{\T}_{j+1}(K))}^2
\le\frac{\Nf\astab}{1+\Nf\astab}|T\tilde\mu_h^\triangle|_{H_\A^1(\triH\backslash\T_{j-1}(K))}^2
\le e^{-\frac{1}{1+\Nf\astab}}|T\tilde\mu_h^\triangle|_{H_\A^1(\triH\backslash\T_{j-1}(K))}^2. 
\]
The theorem then follows from recursive applications of Lemma~\ref{l:decayhc}.
\end{proof}

We now localize $Pv$ since it decays exponentially when $v$ has local support. We consider two families of localizations. The first family $\tilde P^{\triangle,j}$ is based on elements in $\triH$, and utilized to localize $\tilde Tg$, while the second family $P^{\triangle,j}$ is based on faces of $\F_H$ with the purpose to localize $T\lambda_h$.

Fixing $K\in\triH$ and a positive integer $j$, let $\tilde\Lambda_h^{\triangle,K,j}\subset\tilde\Lambda_h^\triangle$ be the set of functions of $\tilde\Lambda_h^\triangle$ vanishing on faces of elements in $\triH\backslash\T_j(K)$, and
$P^{\triangle,K,j}:\VH\to\tilde\Lambda_h^{\triangle,K,j}$ such that, for $v\in\VH$,
\begin{equation}\label{e:Ptriagkj}
  (\tilde\mu_h^{\triangle,K,j},TP^{\triangle,K,j}v)_{\btriH}
  =(\tilde\mu_h^{\triangle,K,j},v_K)_{\btriH}
  \quad\text{for all }\tilde\mu_h^{\triangle,K,j}\in\tilde\Lambda^{\triangle,K,j}_h,
\end{equation}
where we let $v_K$ to be equal to $v$ on $K$ and zero otherwise. We define then $\tilde{P}^{\triangle,j}$ by 
\begin{equation}\label{e:Pjdefminus}
\tilde P^{\triangle,j}v=\sum_{K\in \triH}P^{\triangle,K,j}v_K. 
\end{equation}

We next introduce a new localization, this time based on faces. For a fixed $F\in\F_H$ shared by elements $\tau_1^F$ and  $\tau_2^F$  or shared by only one element $\tau^F$, define $\T_0(F)=\emptyset$, $\T_1(F)=\{\tau^F_1,\tau^F_2\}$ or $\T_1(F)=\{\tau^F\}$, and for  $j=1,2,\dots$ let
\[
\T_{j+1}(F)
=\{\tau\in\triH:\,\bar\tau\cap\bar\tau_j\ne\emptyset\text{ for some }\tau_j\in\T_j(F)\}. 
\]
Let $\tilde\Lambda_h^{\triangle,F,j}\subset\tilde\Lambda_h^\triangle$ be the set of functions of $\tilde\Lambda_h^\triangle$ vanishing on faces of elements in $\triH\backslash\F_j(F)$. Let us decompose $\lambda_h \in \Lambda_h$ into $\lambda_h=\sum_{F\in\F_H}\lambda_h^F$ where $\lambda_h^F=\lambda_h$ on the face $F$ and zero everywhere else on $\F_H$.  The operator $P^{\triangle,F,j}T:\Lambda_h\to\tilde\Lambda_h^{\triangle,F,j}$ is defined by 
\begin{equation}\label{e:PFjdefhc}
(\tilde\mu_h^\triangle,TP^{\triangle,F,j}T\lambda_h)_{\btriH}
=(\tilde\mu_h^\triangle,T\lambda_h^F)_{\partial\tau_1^F}+(\tilde\mu_h^\triangle,T\lambda_h^F)_{\partial\tau_2^F}
\quad\text{for all }\tilde\mu_h^\triangle\in\tilde\Lambda^{\triangle,F,j}_h. 
\end{equation}
We define then $P^{\triangle,j}T\lambda_h\in\tilde\Lambda_h^\triangle$ by 
\begin{equation}\label{e:Pjdefhc}
P^{\triangle,j}T\lambda_h=\sum_{F\in\F_H}P^{\triangle,F,j}T\lambda_h. 
\end{equation}
The reason to introduce $P^{\triangle,j}T\lambda_h$ is because we could not prove that $\tilde P^{\triangle,j}T \tilde\lambda_h^\triangle=\tilde\lambda_h^\triangle$ and this property is fundamental on what follows.
\begin{lemma} \label{e:invariance}
Let $j\ge1$ and $\tilde\lambda_h^\triangle\in\tilde\Lambda_h^\triangle$ and $P^{\triangle,j}$ be defined as above. Then $P^{\triangle,j}T\tilde\lambda_h^\triangle=\tilde\lambda_h^\triangle$.
\end{lemma}
\begin{proof}
Let $F\in\F_H$ and $\lambda_h^F=\lambda_h$ on $F$ and zero everywhere else. Since
\[
(\tilde\mu_h^\triangle,T\tilde\lambda_h^{\triangle,F})_{\partial\tau_1^F}
+(\tilde\mu_h^\triangle,T\tilde\lambda_h^{\triangle,F})_{\partial\tau_2^F}
=(\tilde\mu_h^\triangle,T \tilde\lambda_h^{\triangle,F})_\btriH
\quad\text{for all }\tilde\mu_h^\triangle\in\tilde\Lambda^{\triangle}_h, 
\]
and $\tilde{\lambda}_h^{\triangle,F}\in\tilde\Lambda_h^{\triangle,F,j}$, then it follows from the existence and uniqueness of $P^{\triangle,F,j}T\tilde\lambda_h^\triangle$ that  $P^{\triangle,F,j}T\tilde\lambda_h^{\triangle,F}=\tilde\lambda_h^{\triangle,F}$. Also, $P^{\triangle,F,j}T\tilde\lambda_h^\triangle=P^{\triangle,F,j}T\tilde\lambda_h^{\triangle,F}$, and it then follows from~\eqref{e:Pjdefhc} that
\[
P^{\triangle,j}T\lambda_h^\triangle
=\sum_{F\in\F_H}P^{\triangle,F,j}T\lambda_h^{\triangle,F}
=\sum_{F\in\F_H}\lambda_h^{\triangle,F}
=\lambda_h^\triangle. 
\]
\end{proof}

In Lemmas~\ref{l:errorPtjhc} and~\ref{l:errorPjhc} we estimate the accuracy of $\tilde P^{\triangle,j}$ and $P^{\triangle,j}$ with respect to $P^\triangle$. 

\begin{lemma}\label{l:errorPtjhc}
Consider $v\in H^1(\triH)$, and the operators $P^\triangle$ defined by~\eqref{e:Pdefhigh} and $\tilde P^{\triangle,j}$ by~\eqref{e:Pjdefminus}.
Then
\begin{equation*}
|T(P^\triangle-\tilde P^{\triangle,j})v|_{H_\A^1(\triH)}^2
\le cj^{2d}\Nf^4\astab^2e^{-\frac{[(j-1)/2]}{1+\Nf\astab}}|v|_{H_\A^1(\triH)}^2, 
\end{equation*}
where $\Nf$ is the maximum number of faces of an arbitrary element. 
\end{lemma}
\begin{proof} 
For $K\in\triH$, let $\tilde\mu_h^{\triangle,K}=P^\triangle v_K$ and $\tilde\mu_h^{\triangle,K,j}=P^{\triangle,j,K}v_K$, and
\[
\tilde\psi_h^\triangle=\sum_{K\in\triH}(\tilde\mu_h^{\triangle,K}-\tilde\mu^{\triangle,K,j}_h).
\]
For each $K\in\triH$, let $\tilde\psi^{\triangle,K}_h\in\tilde\Lambda_h^\triangle$ be defined by $\tilde\psi^{\triangle,K}_h|_F=0$ if $F$ is a face of an element of $\T_j(K)$ and $\tilde\psi^{\triangle,K}_h  |_F=\tilde\psi_h^\triangle|_F$, otherwise. We obtain 
\begin{equation} \label{splittinghc}
|T\tilde\psi_h^\triangle|_{H_\A^1(\triH)}^2 
=\sum_{K\in\triH}\sum_{\tau\in\triH}
(\tilde\psi_h^\triangle-\tilde\psi^{\triangle,K}_h,
T(\tilde\mu_h^{\triangle,K}-\tilde\mu_h^{\triangle,K,j}))_{\partial\tau}
+(\tilde\psi^{\triangle,K}_h,T(\tilde\mu_h^{\triangle,K} - \tilde\mu_h^{\triangle,K,j}))_{\partial \tau}. 
\end{equation}
See that the second term of~\eqref{splittinghc} vanishes since 
\[
\sum_{\tau\in\triH}(\tilde\psi^{\triangle,K},T(\tilde\mu_h^{\triangle,K}-\tilde\mu_h^{\triangle,K,j}))_{\partial\tau}
=\sum_{\tau\in\triH}(\tilde\psi^{\triangle,K},T\tilde\mu_h^{\triangle,K})_{\partial \tau} = 0.
\]
For the first term of~\eqref{splittinghc} we use a direct application of~\eqref{e:traceestdelta}, yielding
\begin{multline*}
\sum_{\tau\in\triH}
\bigl(\tilde\psi_h^\triangle-\tilde\psi^{\triangle,K}_h,
T(\tilde\mu_h^{\triangle,K}-\tilde\mu_h^{\triangle,K,j})\bigr)_{\partial\tau}
\\
\le\sum_{\tau\in\T_{j+1}(K)}
|T(\tilde\psi_h^\triangle-\tilde\psi_h^{\triangle,K})|_{H_\A^1(\tau)}
|T(\tilde\mu_h^{\triangle,K}-\tilde\mu_h^{\triangle,K,j})|_{H_\A^1(\tau)}
\\
\le \Nf\astab^{1/2}|T\tilde\psi_h^\triangle|_{H_\A^1(\T_{j+1}(K))}
|T(\tilde\mu_h^{\triangle,K}
-\tilde\mu_h^{\triangle,K,j})|_{H_\A^1(\T_{j+1}(K))}.
\end{multline*}
Let $\nu_h^{\triangle,K,j} \in \tilde\Lambda_h^{\triangle,K,j}$ be equal to zero on all faces of elements of $\T_H\backslash\T_j(K)$
and equal to $\tilde\mu_h^{\triangle,K}$ otherwise. Using Galerkin best approximation property,~\eqref{e:traceestdelta}, and Theorem~\ref{t:decaythmhc} we obtain
\begin{multline*}
|T(\tilde\mu_h^{\triangle,K}-\tilde\mu_h^{\triangle,K,j})|_{H_\A^1({\T}_{j+1}(K))}^2
\le|T(\tilde\mu_h^{\triangle,K}-\tilde\mu_h^{\triangle,K,j})|_{H_\A^1(\triH)}^2
\le|T(\tilde\mu_h^{\triangle,K}-\nu_h^{\triangle,K,j})|_{H_\A^1(\triH)}^2
\\ 
\le \Nf^2\astab|T\tilde\mu_h^{\triangle,K}|_{H_\A^1(\triH\backslash\T_{j-1}(K))}^2 
\le \Nf^2\astab e^{-\frac{[(j-1)/2]}{1+\Nf\astab}}|T\tilde\mu_h^{\triangle,K}|_{H_\A^1(\triH)}^2.
\end{multline*}
We gather the above results to obtain  
\begin{multline*}
|T\tilde\psi_h^\triangle|_{H_\A^1(\triH)}^2
\le \Nf^2\astab e^{-\frac{[(j-1)/2]}{2(1+\Nf\astab)}}
\sum_{K\in\triH}
|T\tilde\psi_h^\triangle|_{H_\A^1({\T}_{j+1}(K))}
|T\tilde\mu_h^{\triangle,K}|_{H_\A^1(\triH)}
\\
\le c\Nf^2\astab e^{-\frac{[(j-1)/2]}{2(1+\Nf\astab)}}j^d
|T\tilde\psi_h^{\triangle}|_{H_\A^1(\triH)}
\biggl(\sum_{K\in\triH}|T\tilde\mu_h^{\triangle,K}|_{H_\A^1(\triH)}^2\biggr)^{1/2}. 
\end{multline*}
We finally gather that 
\[
|T\tilde\mu_h^{\triangle,K}|_{H_\A^1(\triH)}^2
=(\tilde\mu_h^{\triangle,K},TPv_K)_{\partial \triH}
=(\tilde\mu_h^{\triangle,K}, v_K)_{\partial \triH}
=\int_K\A\bgrad(T\tilde\mu_h^{\triangle,K})\cdot\bgrad(v_K)\,d\xx
\]
and from Cauchy--Schwarz, $|T\tilde\mu_h^{\triangle,K}|_{H_\A^1(\triH)}\le|v_K|_{H_\A^1(K)}$,
we have
\[
\sum_{K\in\triH}|T\tilde\mu_h^{\triangle,K}|_{H_\A^1(\triH)}^2
\le|v|_{H^1_\A(\triH)}^2.
\]
\end{proof}

\begin{lemma}\label{l:errorPjhc}
Consider $\lambda_h\in\Lambda_h$, and the operators $P^\triangle$ defined by~\eqref{e:Pdefhigh} and $P^{\triangle,j}$ by~\eqref{e:Pjdefhc}.
Then
\begin{equation*}
|T(P^\triangle-P^{\triangle,j})T\lambda_h|_{H_\A^1(\triH)}^2
\le cj^{2d}\Nf^4\astab^2e^{-\frac{[(j-1)/2]}{1+\Nf\astab}}|T\lambda_h|_{H_\A^1(\triH)}^2
\end{equation*}
\end{lemma}
\begin{proof} 
For $F\in\btriH$, let $\tilde\mu_h^{\triangle,F}=P^\triangle T\lambda_h^F$ and $\tilde\mu_h^{\triangle,F,j}=P^{\triangle,F,j}T\lambda_h^F$, and
\[
\tilde\psi_h^\triangle
=\sum_{F\in\btriH}(\tilde\mu_h^{\triangle,F}-\tilde\mu^{\triangle,F,j}_h).
\]
For each $F\in\btriH$, let $\tilde\psi^{\triangle,F}_h\in\tilde\Lambda_h^\triangle$ be defined by $\tilde\psi^{\triangle,F}_h|_{\hat F}=0$ if $\hat F$ is a face of an element of $\T_j(F)$ and $\tilde\psi^{\triangle,F}_h|_{\hat F}=\tilde\psi_h^\triangle|_{\hat F}$, otherwise. We obtain 
\begin{equation} \label{splittinghcF}
|T\tilde\psi_h^\triangle|_{H_\A^1(\triH)}^2 
=\sum_{\tau\in\triH}\sum_{F\in\btriH}
(\tilde\psi_h^\triangle-\tilde\psi^{\triangle,F}_h,
T(\tilde\mu_h^{\triangle,F}-\tilde\mu_h^{\triangle,F,j}))_{\partial\tau}
+(\tilde\psi^{\triangle,F}_h,T(\tilde\mu_h^{\triangle,F}
-\tilde\mu_h^{\triangle,F,j}))_{\partial \tau}. 
\end{equation}
See that the second term of~\eqref{splittinghcF} vanishes since 
\[
\sum_{\tau\in\triH}(\tilde\psi^{\triangle,F},
T(\tilde\mu_h^{\triangle,F}-\tilde\mu_h^{\triangle,F,j}))_{\partial\tau}
=\sum_{\tau\in\triH}(\tilde\psi^{\triangle,F},T\tilde\mu_h^{\triangle,F})_{\partial \tau} = 0.
\]
For the first term of~\eqref{splittinghcF} we use a direct application of~\eqref{e:traceestdelta}, yielding
\begin{multline*}
\sum_{\tau\in\triH}
\bigl(\tilde\psi_h^\triangle-\tilde\psi^{\triangle,F}_h,
T(\tilde\mu_h^{\triangle,F}-\tilde\mu_h^{\triangle,F,j})\bigr)_{\partial\tau}
\le\sum_{\tau\in\T_{j+1}(F)}
|T(\tilde\psi_h^\triangle-\tilde\psi_h^{\triangle,F})|_{H_\A^1(\tau)}
|T(\tilde\mu_h^{\triangle,F}-\tilde\mu_h^{\triangle,F,j})|_{H_\A^1(\tau)}
\\
\le \Nf\astab^{1/2}|T\tilde\psi_h^\triangle|_{H_\A^1(\T_{j+1}(F))}
|T(\tilde\mu_h^{\triangle,F}
-\tilde\mu_h^{\triangle,F,j})|_{H_\A^1(\T_{j+1}(F))}.
\end{multline*}
Let $\nu_h^{\triangle,K,j}\in\tilde\Lambda_h^{\triangle,K,j}$ be equal to zero on all faces of elements of $\T_H\backslash\T_j(F)$ and equal to $\tilde\mu_h^{\triangle,F}$ otherwise. Using Galerkin best approximation property,~\eqref{e:traceestdelta}, and Theorem~\ref{t:decaythmhc} we obtain
\begin{multline*}
|T(\tilde\mu_h^{\triangle,F}-\tilde\mu_h^{\triangle,F,j})|_{H_\A^1({\T}_{j+1}(K))}^2
\le|T(\tilde\mu_h^{\triangle,F}-\tilde\mu_h^{\triangle,F,j})|_{H_\A^1(\triH)}^2
\le|T(\tilde\mu_h^{\triangle,F}-\nu_h^{\triangle,F,j})|_{H_\A^1(\triH)}^2
\\ 
\le \Nf^2\astab|T\tilde\mu_h^{\triangle,F}|_{H_\A^1(\triH\backslash\T_{j-1}(K))}^2 
\le \Nf^2\astab e^{-\frac{[(j-1)/2]}{1+\Nf\astab}}|T\tilde\mu_h^{\triangle,F}|_{H_\A^1(\triH)}^2.
\end{multline*}
We combine the above results to obtain  
\begin{multline*}
|T\tilde\psi_h^\triangle|_{H_\A^1(\triH)}^2
\le \Nf^2\astab e^{-\frac{[(j-1)/2]}{2(1+\Nf\astab)}}
\sum_{F\in\btriH}
|T\tilde\psi_h^\triangle|_{H_\A^1({\T}_{j+1}(F))}
|T\tilde\mu_h^{\triangle,F}|_{H_\A^1(\triH)}
\\
\le c\Nf^2\astab e^{-\frac{[(j-1)/2]}{2(1+\Nf\astab)}}j^d
|T\tilde\psi_h^\triangle|_{H_\A^1(\triH)}
\biggl(\sum_{F\in\btriH}|T\tilde\mu_h^{\triangle,F}|_{H_\A^1(\triH)}^2\biggr)^{1/2}, 
\end{multline*}
and finally gather that 
\[
|T\tilde\mu_h^{\triangle,F}|_{H_\A^1(\triH)}^2
=(\tilde\mu_h^{\triangle,F},TP^\triangle T\lambda_h^F)_{\partial \triH}
=(\tilde\mu_h^{\triangle,F},T\lambda_h^F)_{\btriH}
=\int_{\tau_1\cup\tau_2}\A\bgrad(T\tilde\mu_h^{\triangle,F})\cdot\bgrad(T\lambda_h^F)\,d\xx. 
\]
From Cauchy--Schwarz and $|T\tilde\mu_h^{\triangle,F}|_{H_\A^1(\triH)}\le|T\lambda_h^F|_{H_\A^1(\tau_1\cup\tau_2)}$ we have
\[
\sum_{F\in\btriH}|T\tilde\mu_h^{\triangle,F}|_{H_\A^1(\triH)}^2\le|T\lambda_h^F|_{H^1_\A(\triH)}^2.
\]
\end{proof}

The LSD method is defined by computing
\begin{equation}\label{e:uhjhighdef}
u_h^{\text{LSD},j}
=u_h^{\text{LSD},0,j}+T \lambda_h^{\text{LSD},j}+\tilde T\bg,
\qquad
\lambda_h^{\text{LSD},j}=\lambdarm+\tlambda_h^{0,\Pi,j}+\tilde\lambda_h^{\triangle,j}, 
\end{equation}
based on modifications of~\eqref{e:tlrmeqtn},~\eqref{e:lambdadef}, and the fourth equation of~\eqref{e:mc3}. Indeed, define $\tlambda_h^{0,\Pi,j} \in
\tilde{\Lambda}_h^{0,\Pi}$ from 
\begin{multline}\label{e:tlrmplusjeqtn} \bigl((I-P^{\triangle,j}T)\tilde\mu_h^{0,\Pi},T(I-P^{\triangle,j}T)\tlambda_h^{0,\Pi,j}\bigr)_{\btriH}
  =-\bigl((I-P^{\triangle,j}T)\tilde\mu_h^{0,\Pi},(I-T\tilde{P}^{\triangle,j})\tilde T\bg\bigr)_{\btriH}
  \\
  -\bigl((I-P^{\triangle,j}T) \tilde\mu_h^{0,\Pi},T(I-P^{\triangle,j}T)\lambdarm  \bigr)_{\btriH}
  \quad\text{for all } \tilde\mu_h^{0,\Pi}\in\tilde\Lambda_h^{0,\Pi},  
\end{multline}
and compute $\tilde\lambda_h^{\triangle,j}\in\tilde\Lambda_h^{\triangle}$ and $u_h^{0,\text{LSD},j}\in\VVRM$ from
\begin{gather}
\tilde\lambda_h^{\triangle,j}
=-P^{\triangle,j}(T\lambdarm+T\tilde\lambda^{0,\Pi,j})-\tilde{P}^{\triangle,j}\tilde T\bg, \label{e:lambdahminusj}
\\
(\murm,u_h^{0,\text{LSD},j})_{\btriH}
=-(\murm,T\lambdarm  +T\tilde\lambda_h^{0,\Pi,j}+T\tilde\lambda_h^{\triangle,j})_{\btriH}
-(\murm,\tilde T\bg)_{\btriH}\quad\text{for all }\murm\in\LRM. \label{e:uzerohigh}
\end{gather}

The next result presents an error estimate of the method.
\begin{theorem}\label{t:erroresthc}
Let $u_h$ be defined by~\eqref{e:alternative} or one of its equivalent formulations~\eqref{e:weak-hybridh},~\eqref{disp-stress2}, and let  $u_h^{\text{LSD},j}$ be as in~\eqref{e:uhjhighdef} for $j\ge1$. Then there exists a constant $c$ such that
\[
|u_h-u_h^{\text{LSD},j}|_{H_\A^1(\triH)}^2
\le cj^{2d}\Nf^4\alpha_{\rm{stab}}^2e^{-\frac{[(j-1)/2]}{1+\Nf\astab}}
\bigl(|T\lambda_h|_{H_\A^1(\triH)}^2 +|\tilde{T}g|_{H_\A^1(\triH)}^2 \bigr).
\]
\end{theorem} 
\begin{proof}
It follows immediately from the definitions of $u_h$ and $u_h^{\text{LSD},j}$ that 
\[
|u_h-u_h^{\text{LSD},j}|^2_{H_\A^1(\triH)}
=|T(\lambda_h-\lambda_h^{\text{LSD},j})|_{H_\A^1(\triH)}^2
=((\lambda_h-\lambda_h^{\text{LSD},j}),T(\lambda_h-\lambda_h^{\text{LSD},j})\bigr)_{\btriH}. 
\]
Defining $\nu_h^{LSD,j}=(I-P^{\triangle,j}T)\lambdarm+(I-P^{\triangle,j}T)\tlambda_h^{0,\Pi}-\tilde P^{\triangle,j}\tilde T\bg$, let 
\begin{multline}\label{e:orthohc}
\bigl((\lambda_h-\lambda_h^{\text{LSD},j}),T(\lambda_h-\lambda_h^{\text{LSD},j})\bigr)_{\btriH}
\\
=\bigl((\lambda_h-\nu_h^{\text{LSD},j}),T(\lambda_h-\lambda_h^{\text{LSD},j})\bigr)_{\btriH}
+\bigl((\nu_h^{\text{LSD},j}-\lambda_h^{\text{LSD},j}),T(\lambda_h-\lambda_h^{\text{LSD},j})\bigr)_{\btriH}.
\end{multline}
Since $\nu_h^{\text{LSD},j}-\lambda_h^{\text{LSD},j}=(I-P^{\triangle,j}T)(\tlambda_h^{0,\Pi}-\tlambda^{0,\Pi,j}_h)$ and $\tlambdarm_h-\tlambda^{0,\Pi,j}_h\in\tilde\Lambda_h^{0,\Pi}$, then by using~\eqref{e:tlrmplusjeqtn} and~\eqref{e:mc3}, respectively, the second term of the right-hand side of~\eqref{e:orthohc} vanishes. Indeed,  from~\eqref{e:uhjhighdef},~\eqref{e:tlrmplusjeqtn},~\eqref{e:lambdahminusj}, 
\begin{equation*}
\bigl((I-P^{\triangle,j}T)\tmurm,T\lambda_h^{\text{LSD},j}\bigr)_{\btriH}
=-\bigl((I-P^{\triangle,j}T)\tmurm,\tilde T\bg\bigr)_{\btriH}
  \quad\text{for all }\tmurm\in\tilde\Lambda_h^{0,\Pi}, 
\end{equation*}
and from~\eqref{e:mc3} we have
\[
\bigl((I-P^{\triangle,j}T)\tmurm,T\lambda_h\bigr)_{\btriH}
=-\bigl((I-P^{\triangle,j}T)\tmurm,\tilde T\bg\bigr)_{\btriH}
\]
since $\tmurm\in\tilde\Lambda_h^0$ and $P^{\triangle,j}T\tmurm\in\tilde\Lambda_h^f$. Thus, choosing $\tmurm=\tlambda_h^{0,\Pi}-\tlambda^{0,\Pi,j}_h$, 
\[
\bigl((\nu_h^{\text{LSD},j}-\lambda_h^{\text{LSD},j}),T(\lambda_h-\lambda_h^{\text{LSD},j})\bigr)_{\btriH}
=\bigl((I-P^{\triangle,j}T)(\tlambda_h^{0,\Pi}-\tlambda_h^{0,\Pi,j}),T(\lambda_h-\lambda_h^{\text{LSD},j})\bigr)_{\btriH}
=0. 
\]

Next, it follows from~\eqref{e:orthohc} and the Cauchy--Schwarz inequality that 
\begin{multline*}
|u_h-u_h^{\text{LSD},j}|_{H_\A^1(\triH)}^2
=\bigl((\lambda_h-\nu_h^{\text{LSD},j}),
T(\lambda_h-\lambda_h^{\text{LSD},j})\bigr)_{\btriH}
=(\lambda_h-\nu_h^{\text{LSD},j},u_h-u_h^{\text{LSD},j})_{\btriH}
\\
\le|T(\lambda_h-\nu_h^{\text{LSD},j})|_{H_\A^1(\triH)}
|u_h-u_h^{\text{LSD},j}|_{H_\A^1(\triH)},
\end{multline*}
since $\lambda_h-\nu_h^{\text{LSD},j}\in\tildeLambda_h$. We now use Lemma~\ref{e:invariance} where $(P^\triangle-P^{\triangle,j})T\tilde\lambda_h^\triangle=0$, and then Lemmas~\ref{l:errorPtjhc} and~\ref{l:errorPjhc} to obtain
\begin{multline}
|u_h-u_h^{\text{LSD},j}|_{H_\A^1(\triH)}^2
\le \bigl(|T(P^\triangle-P^{\triangle,j})T\lambda_h|_{H_\A^1(\triH)}^2
+ T(P^\triangle-\tilde{P}^{\triangle,j})\tilde{T}g|_{H_\A^1(\triH)}^2\bigr)
\\
\le cj^{2d}\Nf^4\alpha_{\rm{stab}}^2e^{-\frac{[(j-1)/2]}{1+\Nf\astab}}
\bigl( |T\lambda_h|_{H_\A^1(\triH)}^2+|\tilde{T}g|_{H_\A^1(\triH)}^2 \bigr).
\end{multline}
\end{proof}
\begin{rem}\label{r:sharpesthc}
We note that $|T\lambda_h|_{H_\A^1(\triH)}^2\le2|u_h|_{H_\A^1(\triH)}^2+2|\tilde{T}g|_{H_\A^1(\triH)}^2$,  therefore
\[
|T\lambda_h|_{H_\A^1(\triH)}^2+|\tilde{T}g|_{H_\A^1(\triH)}^2 
\le4|u-u_h|_{H_\A^1(\triH)}^2+4|u|_{H_\A^1(\triH)}^2+3|\tilde{T}g|_{H_\A^1(\triH)}^2. 
\]
Defining global and local Poincar\'e constants to obtain \label{p:Poincconst}
\[ 
\|u\|_{L^2(\Omega)}\le C_{G,P}|u|_{H_\A^1(\triH)},
\qquad
|\tilde{T}g|_{H_\A^1(\triH)}\le c_P H\|g\|_{L^2(\Omega)},
\]
it follows from Theorem~\ref{t:erroresthc} that if $j$ is taken such that 
\[
cj^{2d}\Nf^4\alpha_{\rm{stab}}^2e^{-\frac{[(j-1)/2]}{1+\Nf^2\alpha_{\rm{stab}}}}(4\HH^2+4C^2_{P,G}+3c^2_pH^2)
\le\HH^2, 
\]
then, from~\eqref{e:HHdef}, 
\begin{equation*}
|u-u_h^{\text{LSD},j}|_{H_\A^1(\triH)}\le\HH\|g\|_{L^2(\Omega)}.
\end{equation*}
We obtain $j=\mathcal{O}(\Nf^2 \astab\log((C_{P,G}+c_pH)/\HH))$. There is a light $\log$ dependence of $j$ on the Poincar\'e constants $C_{P,G}$ and
  $c_p$. 
\end{rem}
\begin{rem}\label{r:equil}
We remark that the post-processed stress $\bsigma_h^{\text{LSD},j}=\A\bgrad[T(\lambdarm+\tlambda_h^{0,\Pi,j}+\tilde\lambda_h^{\triangle,j})+\tilde T\bg]$ is in equilibrium, in the sense that
\begin{equation}\label{e:equil}
\int_\tau\bsigma_h^{\text{LSD},j}\cdot\bgrad v\,d\xx
=\int_\tau\bg v\,d\xx\quad\text{for all } v\in H_0^1(\tau).
\end{equation}
Indeed, given $v\in H_0^1(\tau)$, let $v^0$ be constant and $\tilde v\in H^1(\tau)$ with zero average such that $v=v^0+\tilde v$. Hence
\begin{multline*}
\int_\tau\bg v\,d\xx
=\int_\tau\A\bgrad\tilde T\bg\cdot\bgrad\tilde v\,d\xx+\int_\tau\bg v^0\,d\xx
\\
=\int_\tau[\bsigma_h^{{\text{LSD},j}}-\A\bgrad T(\lambdarm  +\tlambda_h^{0,\Pi,j} +
    \tilde\lambda_h^{\triangle,j})]\cdot\bgrad\tilde v\,d\xx+\int_\tau\bg v^0\,d\xx. 
\end{multline*}
However, $\tilde v=-v^0$ on $\partial\tau$ since $v\in H_0^1(\tau)$, and then 
\begin{multline*}
-\int_\tau\A\bgrad T(\lambdarm  +\tlambda_h^{0,\Pi,j}+\tilde\lambda_h^{\triangle,j})\cdot\bgrad\tilde v\,d\xx+\int_\tau\bg v^0\,d\xx
=-(\lambdarm  +\tlambda_h^{0,\Pi,j}+\tilde\lambda_h^{\triangle,j},\tilde v)_{\partial\tau}+\int_\tau\bg v^0\,d\xx
\\
=(\lambdarm  +\tlambda_h^{0,\Pi,j}+\tilde\lambda_h^{\triangle,j},v^0)_{\partial\tau}+\int_\tau\bg v^0\,d\xx
=0, 
\end{multline*}
where the final estimate follows from~\eqref{e:weak-e1}. Thus,~\eqref{e:equil} holds. 
\end{rem}

\subsection{Decaying Low-Contrast} \label{ss:lowcon} 
The method for the low-contrast case is a particular case of the high-contrast one by choosing
\[
\astab>\alpha.
\]
It then follows from Lemma~\ref{l:eigenbound} that $1\le\alpha_i^F\le\alpha$ for all the local eigenvalues. Thus, $\tilde\Lambda_h^\triangle=\tildeLambda_h^f$ and $\tilde\Lambda_h^\Pi=\emptyset$ from~\eqref{e:Lpmdef} and~\eqref{e:ltpmdef}.

Of course, the numerical scheme and the estimates developed in Section~\ref{ss:dhc} hold. However, several simplifications are possible when the coefficients have low-contrast, leading to sharper estimates. We remark that in this case, our method is similar to that of~\cite{MR3591945}, with some differences. First we consider that $\tilde T$ can be nonzero. Also, our scheme is defined by a sequence of elliptic problems, avoiding the annoyance of saddle point systems. We had to reconsider the proofs, in our view simplifying some of them.


We now establish the following fundamental result for low-contrast. The technique used for the proof is extended in Lemma~\ref{l:decayhc} for the high contrast case. 
The Lemma~\ref{l:decayhc} is now replaced by the following.
\begin{customlemma}{\ref{l:decayhc}'}\label{l:decay}
Let $v\in\VH$ where $\supp v\subset K\in\triH$, and $\tilde\mu_h^f=Pv$. Then, for any integer $j\ge1$, 
\begin{equation*}
|T\tilde\mu_h^f|_{H_\A^1(\triH\backslash\T_{j+1}(K))}^2
\le \Nf\alpha |T\tilde\mu_h^f|_{H_\A^1(\T_{j+1}(K)\backslash\T_j(K))}^2, 
\end{equation*}
where $\alpha$ is defined in Lemma~\ref{l:eigenbound}.
\end{customlemma}
\begin{proof} 
Choose $\tilde\nu_h^f \in\tilde\Lambda_h^f$ defined by $\tilde\nu_h^f|_F=0$ if $F$ is a face of an element of ${\T}_{j}(K)$ and $\tilde\nu_h^f|_F= \tilde\mu_h^f|_F$ otherwise. We obtain, as in the proof of Lemma~\ref{l:decayhc}, 
\begin{equation*}
|T\tilde\mu_h^f|_{H_\A^1(\triH\backslash{\T}_{j+1}(K))}^2
\le\sum_{\tau\in{\T}_{j+1}(K)\backslash{\T}_{j}(K)}(\tilde\mu_h^f -
\tilde\nu_h^{f},T\tilde\mu_h^f)_{\partial\tau}. 
\end{equation*}

Next, let $F$ be a face of an element $\tau\in\T_{j+1}(K) \backslash\T_j(K)$ and let  $\chi_F$ be the characteristic function of $F$ being identically equal to one on $F$ and zero on $\partial\tau\backslash F$. See that $\chi_F(\tilde\mu_h^f-\tilde\nu_h^f)$ vanishes for faces $F$ on $\tau\in\T_{j+1}(K)\backslash\T_j(K)$ that are not shared by an element in $\T_j(K)$ however it is not known a priori how many. Let $\hat\F(\tau)$ be the set of faces of $\tau$ that are also faces of an element in $\T_j(K)$. For the shared faces, $\chi_F(\tilde\mu_h^f-\tilde\nu_h^f)=\chi_F\tilde\mu_h^f$. Since it is possible that all $d$ faces of $\tau$ share faces of elements of $\T_{j}$, hence, the following bound always holds: 
\[
|T (\tilde\mu_h^f-\tilde\nu_h^f)|^2_{H_\A^1(\tau)} \leq
\sum_{F\in\hat\F(\tau)} |T(\chi_F\tilde\mu_h^f)|^2_{H_\A^1(\tau)}.   
\]
It follows from~\eqref{e:traceest} that $|T(\chi_F\tilde\mu_h^f)|_{H_\A^1(\tau)}^2\le\alpha |T\tilde\mu_h^f|^2_{H_\A^1(\tau)}$ and the final result holds since $\hat\F(\tau)$ has at most $\Nf$ faces.
\end{proof}
Comparing Lemma~\ref{l:decay} with Lemma~\ref{l:decayhc}, we see that the right hand side of the estimate depends on $\T_{j+1}(K)\backslash\T_j(K))$ instead of $\T_{j+1}(K)\backslash\T_{j-1}(K))$. This more local versions is possible due to~\eqref{e:traceest} instead of~\eqref{e:traceestdelta}. The result of Lemma~\ref{l:decay} allows a better decay estimate with respect to $j$ in Theorem~\ref{t:decaytheorem} compared to that of Theorem ~\ref{t:decaythmhc} (cf. also Theorems~\ref{t:errorest} and~\ref{t:erroresthc}). However, in general, the estimate is not as good since $\alpha$ increases with the local contrast; see Lemma~\ref{l:eigenbound}. 

\begin{rem} Even though the proof of Lemma~\ref{l:decay} concentrates the analysis to hybrid discretization and $u_h \in H^1(\triH)$, the analysis also extends easily to mixed finite element discretizations or to finite dimensional approximations of $H^1(\triH)$. We also could have used the flux approach for the proof, that is, let $\bsigma^F=\A\bgrad T\mu_F$, $\bsigma_\I^F=\bgrad T_1\mu_F$, $\bsigma=\A\bgrad T\tilde\mu_h^f$ and $\bsigma_\I=\bgrad T_\I\tilde\mu_h^f$, or the corresponding ones arising from the lower-order Raviart--Thomas case (associate to a triangulation ${\mathcal T}_h(\tau)$). We would have 
\begin{multline*}
  \|\A^{-1/2}\bsigma^F\|^2_{L^2(\tau)}
  \le\frac{1}{\amin^\tau}\|\bsigma^F_\I\|^2_{L^2(\tau)}
  \le\frac{{c_1}}{\amin^\tau}|\mu_F|^2_{H^{-1/2}(\partial \tau)}
  \le\frac{{c_2}}{\amin^\tau} \beta_{H/h}^2 |\tilde\mu_h^f|^2_{H^{-1/2}(\partial\tau)}
  \\
  \le\frac{{\gamma}}{\amin^\tau} \beta_{H/h}^2\|\bsigma_\I\|^2_{L^2(\tau)}
  \le {\gamma}\frac{\amax^\tau}{\amin^\tau}\beta_{H/h}^2\|\A^{-1/2}\bsigma\|^2_{L^2(\tau)}. 
\end{multline*}
\end{rem}

\begin{rem} \label{rhochoice} 
We note that ~\cite{MR1759911}*{Lemma 4.4} is based on $H^{-1/2}(\tau)$ norms and therefore it holds whether we use $H^1(\tau), \Hdivtau$ or corresponding discretized versions inside $\tau$. We point out that the $h$ in $\log(H/h)$ is related to the space $\Lambda_h$, not to the interior. The $\alpha$ in this paper is estimated as the worst case scenario, that is, using Lemma \ref{l:extensions} and  ~\cite{MR1759911}*{Lemma 4.4}. For particular cases of coefficients $\A$ and discretizations for $H^1(\tau)$ or $\Hdivtau$, sharper estimated for $\alpha$ can  be derived using weighted and generalized Poincar\'e inequalities techniques and partitions of unity that conform with $\A$ in order to avoid large energies on the interior extensions~\cites{MR1974504,MR2047078,MR2867661,MR2317926,MR1921914,MR3013465,MR3047947,MR2456834,MR2810804,MR3552482,MR2861254}.
\end{rem}   

Following a similar analysis, we next replace Theorem~\ref{t:decaythmhc} with $\astab$ instead of $\alpha$.

\begin{customtheorem}{\ref{t:decaythmhc}'}\label{t:decaytheorem}
Let $v\in\VH$ such that $\supp v\subset K$, and $\tilde\mu_h^f=Pv$. Then, for any integer $j\ge1$, 
\begin{equation}\label{e:decaytildel}
|T\tilde\mu_h^f|_{H_\A^1(\triH\backslash\T_{j+1}(K))}^2
\le e^{-\frac j{1+\Nf\alpha}} |T\tilde\mu_h^f|_{H_\A^1(\triH)}^2.
\end{equation}
\end{customtheorem}
\begin{proof} 
If $\tilde\mu_h^f=Pv$ where $\supp v\subset K$, then using Lemma~\ref{l:decay} we have 
\begin{equation*}
|T\tilde\mu_h^f|_{H_\A^1(\triH\backslash\T_{j+1}(K))}^2
\le \Nf\alpha|T\tilde\mu_h^f|_{H_\A^1(\triH\backslash\T_j(K))}^2
-\Nf\alpha|T\tilde\mu_h^f|_{H_\A^1(\triH\backslash\T_{j+1}(K))}^2
\end{equation*}
and then
\[
|T\tilde\mu_h^f|_{H_\A^1(\triH\backslash\T_{j+1}(K))}^2
\le\frac{\Nf\alpha}{1+\Nf\alpha}|T\tilde\mu_h^f|_{H_\A^1(\triH\backslash\T_j(K))}^2
\le e^{-\frac1{1+\Nf\alpha}}|T\tilde\mu_h^f|_{H_\A^1(\triH\backslash\T_j(K))}^2,
\]
and the theorem follows.
\end{proof} 

Note that if the coefficient $\A$ is nearly constant and isotropic inside each element $\tau$, the exponential decay will depend only mildly on $\beta_{H/h}$. However, the decay of $Pv$ deteriorates as the contrast $\kappa$ gets larger. That is the reason for developing a contrast-free method in Section~\ref{ss:dhc}, where we consider high contrast and eliminate the $\alpha$ dependence. 

For each fixed element $K$ and positive integer $j$, let $\tilde\Lambda_h^{f,K,j}\subset\tilde\Lambda_h^f$ 
be the set of functions of  $\tilde\Lambda_h^f$ which vanish on faces of elements in $\triH\backslash\T_j(K)$. We introduce the operator $P^{K,j}:\VH\to\tilde\Lambda_h^{f,K,j}$ such that, for $v\in\VH$, 
\begin{equation}\label{e:operPkfdef}
(\tilde\mu_h^f,TP^{K,j}v)_{\btriH}=(\tilde\mu^f_h,v_K)_{\btriH}
  \quad\text{for all }\tilde\mu_h^f\in\tilde\Lambda^{f,K,j}_h.
\end{equation}
where $v_K$ is equal to $v$ on $K$ and zero otherwise. For $v\in H^1(\triH)$ we define then $\tilde{P}^jTv\in\tilde\Lambda_h^f$ by
\begin{equation}\label{e:PKjdef}
\tilde P^jv=\sum_{K\in \triH}P^{K,j}v_K. 
\end{equation} \\
We next introduce a localization based on faces. Let $\tilde\Lambda_h^{f,F,j}\subset\tilde\Lambda_h^f$ \label{p:tltfdef} be the set of functions of $\tilde\Lambda_h^f$ vanishing on faces of elements in
$\triH \backslash \F_j(F)$. Let us decompose $\lambda_h \in \Lambda_h$ into
$\lambda_h = \sum_{F \in \F_H} \lambda_h^F$ where  $\lambda_h^F = \lambda_h$ on the
face $F$ and zero everywhere else on $\F_H$. The operator $P^{F,j}:\Lambda_h\to\tilde\Lambda_h^{f,F,j}$ is defined as
\begin{equation}\label{e:PFjdef}
(\tilde\mu_h^f,TP^{F,j}T\lambda_h)_{\btriH}
=(\tilde\mu_h^f,T\lambda_h^F)_{\partial\tau_1^F}+(\tilde\mu_h^f,T\lambda_h^F)_{\partial\tau_2^F}
\quad\text{for all }\tilde\mu_h^f\in\tilde\Lambda^{f,F,j}_h. 
\end{equation}
We define then ${P}^jT\lambda_h\in\tilde\Lambda_h^f$ by 
\begin{equation}\label{e:Pjdef}
P^jT\lambda_h=\sum_{F\in\F_H}P^{F,j}T \lambda_h. 
\end{equation}
\begin{customlemma}{\ref{l:errorPtjhc}'}\label{l:errorPj}
Consider $v\in H^1(\triH)$, and the operators $P$ defined by~\eqref{e:Pdef} and $\tilde P^j$ by~\eqref{e:PKjdef}. Then
\begin{equation}\label{e:final}
|T(P-\tilde{P}^j)v|_{H_\A^1(\triH)}^2
\le c\Nf^4j^{2d}\alpha^2e^{-\frac{j-2}{1+\Nf\alpha}}|v|^2_{H_\A^1(\triH)}.
\end{equation}
\end{customlemma}
\begin{proof} 
We follow similar arguments of the proof of Lemma~\ref{l:errorPtjhc}. For $K\in\triH$, let $\tilde\mu_h^{f,K} = Pv_K$ and $\tilde\mu_h^{f,K,j}=P^{K,j}v_K$, and $\tilde{\psi}_h^f  =\sum_{K\in\triH}(\tilde\mu_h^{f,K}-\tilde\mu^{f,K,j}_h)$. Let $\tilde\psi^{f,K}_h\in\tilde\Lambda_h^f$ be defined by $\tilde\psi^{f,K}_h|_F=0$ if $F$ is a face of an element of $\T_j(K)$ and $\tilde\psi^{f,K}_h|_F=\tilde\psi_h^f|_F$, otherwise. Then 
\begin{equation} \label{splitting}
|T\tilde\psi_h^f|_{H_\A^1(\triH)}^2 
=\sum_{K \in\triH} \sum_{\tau \in\triH} 
(\tilde\psi_h^f-\tilde\psi^{f,K}_h, T(\tilde\mu_h^{f,K}-\tilde\mu_h^{f,K,j}))_{\partial\tau}
+(\tilde\psi^{f,K}_h, T(\tilde\mu_h^{f,K} - \tilde\mu_h^{f,K,j}))_{\partial \tau}. 
\end{equation}
See that the second term of~\eqref{splitting} vanishes since 
\[
\sum_{\tau\in\triH}(\tilde\psi^{f,K},T(\tilde\mu_h^{f,K}-\tilde\mu_h^{f,K,j}))_{\partial\tau}
=\sum_{\tau\in\triH}(\tilde\psi^{f,K},T\tilde\mu_h^{f,K})_{\partial\tau}=0.
\]
For the first term of~\eqref{splitting}, as in Lemma~\ref{l:decay} we use~\eqref{e:traceest}, yielding
\begin{multline*}
\sum_{\tau\in\triH}\bigl(\tilde\psi_h^f-\tilde\psi^{f,K}_h,T(\tilde\mu_h^{f,K}-\tilde\mu_h^{f,K,j})\bigr)_{\partial\tau}
\le\sum_{\tau\in\T_{j+1}(K)}
|T(\tilde\psi_h^f-\tilde\psi_h^{f,K})|_{H_\A^1(\tau)}
|T(\tilde\mu_h^{f,K}-\tilde\mu_h^{f,K,j})|_{H_\A^1(\tau)}
\\
\le \Nf\alpha^{1/2}|T\tilde\psi_h^f|_{H_\A^1(\T_{j+1}(K))}|T(\tilde\mu_h^{f,K}
-\tilde\mu_h^{f,K,j})|_{H_\A^1(\T_{j+1}(K))}.
\end{multline*}
Let $\nu_h^{f,K,j} \in \tilde\Lambda_h^{f,K,j}$ be equal to zero on all faces of elements of $\T_H\backslash\T_j(K)$
and equal to $\tilde\mu_h^{f,K}$ otherwise. Using Galerkin best approximation property,~\eqref{e:traceest} and Theorem~\ref{t:decaytheorem} we obtain
\begin{multline*}
|T(\tilde\mu_h^{f,K}-\tilde\mu_h^{f,K,j})|_{H_\A^1({\T}_{j+1}(K))}^2
\le|T(\tilde\mu_h^{f,K}-\tilde\mu_h^{f,K,j})|_{H_\A^1(\triH)}^2
\le|T(\tilde\mu_h^{f,K}-\nu_h^{f,K,j})|_{H_\A^1(\triH)}^2
\\ 
\le \Nf^2\alpha|T\tilde\mu_h^{f,K}|_{H_\A^1(\triH\backslash\T_{j-1}(K))}^2 
\le \Nf^2\alpha e^{-\frac{j-2}{1+\Nf\alpha}}|T\tilde\mu_h^{f,K}|_{H_\A^1(\triH)}^2.
\end{multline*}
We gather the above results to obtain  
\begin{multline*}
|T\tilde\psi_h^{f,K}|_{H_\A^1(\triH)}^2
\le \Nf^2\alpha e^{-\frac{j-2}{2(1+\Nf\alpha)}}
\sum_{K\in\triH}|T\tilde\psi_h^{f,K}|_{H_\A^1({\T}_{j+1}(K))}|T\tilde\mu_h^{f,K}|_{H_\A^1(\triH)}
\\
\le \Nf^2\alpha e^{-\frac{j-2}{2(1+\Nf\alpha)}}cj^d|T\tilde\psi_h^{f,K}|_{H_\A^1(\triH)}
\biggl(\sum_{K\in\triH}|T\tilde\mu_h^{f,K}|_{H_\A^1(\triH)}^2\biggr)^{1/2}. 
\end{multline*}

We finally gather that 
\[
|T\tilde\mu_h^{f,K}|_{H_\A^1(\triH)}^2
=(\tilde\mu_h^{f,K},TPv_K)_{\partial \triH}
=(\tilde\mu_h^{f,K}, v_K)_{\partial \triH}
=\int_K\A\bgrad(T\tilde\mu_h^{f,K})\cdot\bgrad(v_K)\,d\xx
\]
and from Cauchy--Schwarz, $|T\tilde\mu_h^{f,K}|_{H_\A^1(\triH)}\le|v_K|_{H_\A^1(K)}$,
we have
\[
\sum_{K\in\triH}|T\tilde\mu_h^{f,K}|_{H_\A^1(\triH)}^2\le|v|_{H^1_\A(\triH)}^2.
\]
\end{proof}
\begin{rem}\label{tildePjdecaying}
Using arguments as in the proof of Lemma~\ref{l:errorPjhc} with estimate~\eqref{e:traceest}, we obtain
\begin{equation}\label{e:final2}
|T(P-P^j)T \lambda_h|_{H_\A^1(\triH)}^2
\le cj^{2d}\alpha^2e^{-\frac{j-2}{1+\Nf\alpha}}|T\lambda_h|^2_{H_\A^1(\triH)}
\end{equation}
for $\lambda_h\in\Lambda_h$. 
\end{rem}
Recall from~\eqref{e:lambdadef} and~\eqref{disp-stress2} that $u_h=\uurm_h+T\lambda_h+\tilde T\bg$, where 
\[
\lambda_h=(I-PT)\lambdarm+(I-PT)\tlambdarm_h-P\tilde T\bg.
\]
Motivated by the above decaying results and~\eqref{e:mc3}--\eqref{e:tlrmeqtn}, we define the solution of the localized algorithm by 
\begin{equation}\label{e:uhjdef}
u_h^j=u_h^{0,j}+T\lambda_h^j+\tilde T\bg,
\quad
\text{where }\lambda_h^j=(I-P^jT)\lambdarm+(I-P^jT)\tlambdarmj_h-\tilde{P}^j\tilde T\bg,
\end{equation}
and $\lambdarm$ solves the first equation of~\eqref{e:mc3}. Also, similarly to~\eqref{e:tlrmplusjeqtn}, $\tlambdarmj_h\in\LTRM$ solves
\begin{multline}\label{e:tlrmjeqtn}
  \bigl((I-P^jT)\tmurm,T(I-P^jT)\tlambdarmj_h\bigr)_{\btriH}
  =-\bigl((I-P^jT)\tmurm,(I-T\tilde{P}^j)\tilde T\bg\bigr)_{\btriH}
  \\
  -\bigl((I-P^jT)\tmurm,T(I-P^jT)\lambdarm\bigr)_{\btriH}
  \quad\text{for all }\tmurm\in\LTRM, 
\end{multline}
and similarly to the fourth equation of~\eqref{e:mc3}, we obtain $u_h^{0,j}$ by 
\[
(\murm,u_h^{0,j})_{\btriH}
=-(\murm,T\lambdarm+T\tlambdarmj_h+T\tilde\lambda_h^{f,j})_{\btriH}
-(\murm,\tilde T\bg)_{\btriH}\quad\text{for all }\murm\in\LRM. 
\]
and as in~\eqref{e:lambdadef} we have defined 
\begin{equation}\label{e:lambdahfj}
\tilde\lambda_h^{f,j}=-P^j(T\lambdarm+T\tlambdarmj_h)-\tilde{P}^j\tilde T\bg.
\end{equation}
A fundamental difference between our discretization and some
multiscale mixed finite elements methods such as the ones in
~\cites{MR2306414,MR2831590} is that here we avoid solving a saddle point problem by computing $\lambda_h^{0,j}=\lambda^0$ using the first equation of~\eqref{e:mc3}, while there the equation~\eqref{e:tlrmjeqtn} is extended to the whole space $\Lambda^0\oplus\LTRM$ and solved together with an equation for $u_h^{0,j}$. 

A new version of Theorem~\ref{t:erroresthc} follows.
\begin{customtheorem}{\ref{t:erroresthc}'}\label{t:errorest}
Let $u_h$ and $u_h^j$ be defined by~\eqref{disp-stress2} and~\eqref{e:uhjdef}. Then there exists a constant $c$ such that
\begin{equation*}
|u_h-u_h^j|^2_{H_\A^1(\triH)}
\le cj^{2d}\Nf^4\alpha^2e^{-\frac{j-2}{1+\Nf\alpha}}
\bigl(|T\lambda_h|_{H_\A^1(\triH)}^2+|\tilde{T}g|_{H_\A^1(\triH)}^2\bigr).
\end{equation*}
\end{customtheorem} 
\begin{proof}
It follows immediately from the definitions of $u_h$ and $u_h^j$ that 
\[
|u_h-u_h^j|^2_{H_\A^1(\triH)}
=|T(\lambda_h-\lambda_h^j)|_{H_\A^1(\triH)}^2
=((\lambda_h-\lambda_h^j),T(\lambda_h-\lambda_h^j)\bigr)_{\btriH}. 
\]
Defining $\nu_h^j=(I-P^jT)\lambdarm+(I-P^jT)\tlambdarm_h-\tilde{P}^j\tilde T\bg$, it follows that 
\begin{equation}\label{e:ortho}
  \bigl((\lambda_h-\lambda_h^j),T(\lambda_h-\lambda_h^j)\bigr)_{\btriH}
  =\bigl((\lambda_h-\nu_h^j),T(\lambda_h-\lambda_h^j)\bigr)_{\btriH}
  +\bigl((\nu_h^j-\lambda_h^j),T(\lambda_h-\lambda_h^j)\bigr)_{\btriH}.
 \end{equation}
Since $\tlambdarm_h-\tlambdarmj_h \in \tilde\Lambda_h^0$ and $\nu_h^j-\lambda_h^j=(I-P^jT)(\tlambdarm_h-\tlambdarmj_h)\in\tilde\Lambda_h$ then by using~\eqref{e:tlrmjeqtn} and~\eqref{e:mc3}, respectively, the second term of the right-hand side of~\eqref{e:ortho} vanishes. Indeed, from~\eqref{e:uhjdef},~\eqref{e:tlrmjeqtn},~\eqref{e:lambdahfj}, 
\begin{equation*}
  \bigl((I-P^jT)\tmurm,T\lambda_h^j\bigr)_{\btriH}
  =-\bigl((I-P^jT)\tmurm,\tilde T\bg\bigr)_{\btriH}
  \quad\text{for all }\tmurm\in\LTRM, 
\end{equation*}
and from~\eqref{e:mc3} we have
\[
\bigl((I-P^jT)\tmurm,T\lambda_h\bigr)_{\btriH}
=-\bigl((I-P^jT)\tmurm,\tilde T\bg\bigr)_{\btriH}
\]
since $\tmurm \in \tilde\Lambda_h^0$ and $P^jT\tmurm \in \tilde\Lambda_h^f$. Thus, choosing $\tmurm=\tlambdarm_h-\tlambdarmj_h$,
\[
\bigl((\nu_h^j-\lambda_h^j),T(\lambda_h-\lambda_h^j)\bigr)_{\btriH}
=\bigl((I-P^jT)(\tlambdarm_h-\tlambdarmj_h),T(\lambda_h-\lambda_h^j)\bigr)_{\btriH}
=0. 
\]

By Cauchy--Schwarz inequality, we have from~\eqref{e:ortho} that 
\begin{multline*}
|u_h-u_h^j|_{H_\A^1(\triH)}^2
=\bigl((\lambda_h-\nu_h^j),T(\lambda_h-\lambda_h^j)\bigr)_{\btriH}
=(\lambda_h-\nu_h^j,u_h-u_h^j)_{\btriH}
\\
\le|T(\lambda_h-\nu_h^j)|_{H_\A^1(\triH)}|u_h-u_h^j|_{H_\A^1(\triH)},
\end{multline*}
since $\lambda_h-\nu_h^j\in\tildeLambda_h$. We now use Lemma~\ref{e:invariance} where $(P-P^j)T\tilde\lambda_h^f=0$ and then Lemma~\ref{l:errorPj} and Remark~\ref{tildePjdecaying} to obtain
\begin{multline*}
|u_h-u_h^j|_{H_\A^1(\triH)}^2
\le |T(\lambda_h-\nu_h^j)|_{H_\A^1(\triH)}^2
= |T(P-P^j)T\lambda_h +T(P-\tilde{P}^j)\tilde{T}g|_{H_\A^1(\triH)}^2
\\
\le cj^{2d}\Nf^4\alpha^2e^{-\frac{j-2}{1+\Nf\alpha}}
\bigl(|T\lambda_h|_{H_\A^1(\triH)}^2+|\tilde Tg|_{H_\A^1(\triH)}^2\bigr).
\end{multline*}
\end{proof}
\begin{rem}\label{r:sharpest}
Arguing as remark~\ref{r:sharpesthc}, we gather from Theorem~\ref{t:errorest} that if $j$ is taken such that 
\[
cj^{2d}\Nf^4\alpha^2e^{-\frac{j-2}{1+\Nf\alpha}}(4\HH^2+4C^2_{P,G}+3c^2_pH^2)
\le\HH^2, 
\]
then, 
\begin{equation*}
|u-u_h^{\text{LSD},j}|_{H_\A^1(\triH)}\le\HH\|g\|_{L^2(\Omega)}.
\end{equation*}

We obtain $j=\mathcal{O}(\Nf^2\kappa\log((C_{P,G}+c_pH)/\HH)\log^2(H/h))$, and $j$ is large not only in the high contrast case ($\kappa$ large), but also if $h\ll H$. 
\end{rem}

\section{Finite Dimensional Right-Hand Side}\label{ss:findim}
We first note that the solution given by the four steps of~\eqref{e:mc3} is exact, in the sense that it solves~\eqref{e:weak-hybridh} as well, and hence the solution error vanishes. To compute $\tlambdarm_h$ and $\tilde\lambda_h^f$ is necessary to compute $P\tilde T\bg$, and that hampers efficiency, and it is 
the only part of algorithm that does not permit exact pre-processing.  
 In order to allow pre-processing, we replace the space $L^2(\Omega)$, which contains $\bg$, by a finite dimensional one in such a way that the solution error is $O(\HH)$. If $v_i$ is a basis function of this finite dimensional space, $P\tilde Tv_i$ can be built in advance as a pre-processing computation. To guarantee order $\HH$ convergence in the energy norm, it is enough to define the basis using elementwise generalized eigenvalues problems. For each $\tau\in \triH$, find eigenpairs $(\sigma_i,v_i) \in (\RR,H^1(\tau))$, $i\in\NN$, such that 
\begin{equation}\label{e:eigenproblem}
\int_\tau\A\bgrad v_j\cdot\bgrad w\,d\xx
=\sigma_j\int_\tau v_iw\,d\xx\quad\text{for all }w\in H^1(\tau)  
\end{equation}
Let us order the eigenvalues $0=\sigma_1<\sigma_2\le\sigma_3 \le\dots$.
 First note that $\sigma_i\to\infty$ since $H^1(\tau)$ is compactly embedded in $L^2(\tau)$ and $\A$ is  uniformly positive definite and bounded. Define $J_\tau$ as the minimum integer such that $\sigma_{J_\tau+1}^{-1}\le c\HH^2$, where $\HH$ was introduced in~\eqref{e:HHdef},
and define the space 
\begin{equation*}\label{e:eigenspace}
\mathcal F_{J_\tau}=\spann\{v_1,\dots,v_{J_\tau}\}. 
\end{equation*}
Then we obtain 
\[
v\in{\mathcal F_{J_\tau}^{\perp_{L^2}}}
\implies\int_\tau v^2\,d\xx
\le c\HH^2\int_\tau\A\bgrad v\cdot\bgrad v\,d\xx.
\]
Indeed, let $v\in{\mathcal F_{J_\tau}^{\perp_{L^2}}}$. Using the fact that the eigenfunctions $v_{\tau,i}$ define an orthogonal basis in both $H_\A^1(\tau)$ semi-norm and $L^2(\tau)$ norm, we can write $v=\sum_{i\ge J_\tau+1}\alpha_iv_i$, and then
\begin{multline}\label{e:Poincarehetero}
\int_\tau v^2\,d\xx
=\sum_{i\ge J_\tau+1}\alpha_i^2\int_\tau v_i^2\,d\xx
=\sum_{i\ge J_\tau+1}\alpha_i^2\sigma_i^{-2}\int_\tau\A\bgrad v_i \cdot\bgrad v_i \,d\xx \le
\\
\sigma_{J_\tau+1}^{-2}\sum_{i\ge J_\tau+1}\alpha_i^2\int_\tau\A\bgrad v_i \cdot \bgrad v_i 
\,d\xx
\le (c\HH)^2\sum_{i\ge J_\tau+1}\alpha_i^2\int_\tau \A\bgrad v_i \cdot \bgrad v_i\,d\xx
\\
=(c\HH)^2\int_\tau \A\bgrad v \cdot\bgrad v \,d\xx.
\end{multline}
Clearly, $\mathcal F_{J_\tau}$ is nonempty since it contains the constant function.

Let
\begin{equation}\label{e:Fjdef}
\mathcal F_J=\{v\in\VH:\,v|_\tau\in\mathcal F_{J_\tau}\}, 
\end{equation}
and let $\Pi_J\bg\in\mathcal F_J$ be the $L^2(\Omega)$ orthogonal projection of $\bg$ on $\mathcal F_{J}$. \label{p:Pidef}
Note that the computation of $\tilde T\bg$ becomes now trivial since, on each element $\tau$ and $i>1$, $\tilde Tv_i=\sigma_i^{-1}v_i$. This follows from the second equation of~\eqref{e:definitionT}, and~\eqref{e:eigenproblem}. 

\begin{lemma}\label{l:exactappro}
Consider $u_J\in H_0^1(\Omega)$ weakly satisfying $-\div\A\bgrad u_J=\Pi_J\bg$. Then 
\[
|u-u_J|_{H_\A^1(\triH)}\le c\HH\|\bg\|_{L^2(\Omega)}=c\HH\|f\|_{L^2(\Omega)}.
\]
\end{lemma} 
\begin{proof} 
Note that the error 
$e_J=u-u_J\in H_0^1(\Omega)$ weakly satisfies $-\div\A\bgrad e_J=(I-\Pi_J)\bg$, and note that $u_J$ does not necessary belong to $\mathcal F_{J}$. We have
\begin{multline*}
  |e_J|_{H_\A^1(\triH)}^2
  =(\A\bgrad e_J,\bgrad e_J)_\triH
=((I-\Pi_J)\bg,e_J)_\triH
=((I-\Pi_J)\bg,e_J-\Pi_J e_J)_\triH
\\
\le\|\bg\|_{L^2(\Omega)}\|e_J-\Pi_J e_J\|_{L^2(\Omega)}
    \le c\HH\|\bg\|_{L^2(\Omega)} |e_J-\Pi_J e_J|_{H_\A^1(\triH)} 
\le c\HH\|\bg\|_{L^2(\Omega)}|e_J|_{H_\A^1(\triH)},
\end{multline*}
where we have used that  $\Pi_J e_J \in \mathcal F_J$, and $\Pi_J$ is 
an orthogonal projection in both $L^2(\Omega)$ and $H_\A^1(\triH)$ inner-products.
\end{proof}

\begin{rem}
The above arguments can be extended to the discrete case $u_h$ and $u_{h,J}$ solutions of~\eqref{e:mc3} with $\bg$ and $\Pi_J\bg$ as the forcing term, respectively. Denote $\delta\bg=(I-\Pi_J)\bg$. The solution error $e_{h,J}=u_h-u_{h,J}$ can be decomposed as $e_{h,J}=e_{h,J}^0+T\delta\lambda_h+\tilde T\delta\bg$, where $\delta\lambda_h=(I-PT)\delta\tilde\lambda^0-P\tilde T\delta\bg$. 
Note that $\delta\lambda^0$ vanishes. Using the definitions of $T$ and $\tilde{T}$, and~\eqref{e:weak-hybridh} we obtain 
\begin{equation}\label{e:ehJerror}
|e_{h,J}|_{H_\A^1(\triH)}^2
=|T\delta\lambda_h+\tilde T\delta\bg|_{H_\A^1(\triH)}^2
=(e_{h,J},\delta\bg)_{L^2(\Omega)}
\le c \HH|e_{h,J}|_{H_\A^1(\triH)}\|\bg\|_{L^2(\Omega)}.
 \end{equation}
\end{rem}

\begin{rem}
On Section~\ref{s:Hlocal} we have assumed $\bg$ as the right-hand side. In case we use $\Pi_J\bg$ as the right hand side, we obtain the same bound since $\|\Pi_J\bg\|_{L^2(\Omega)} \le\|\bg\|_{L^2(\Omega)}$. Thus, if $u_{h,J}^j$ is the localized solution of the method with $\Pi_J\bg$ as the right-hand side , then
  \[
  \|u_h-u_{h,J}^j\|_{H_\A^1(\triH)}
  \le\|u_h-u_{h,J}\|_{H_\A^1(\triH)}+\|u_{h,J}-u_{h,J}^j\|_{H_\A^1(\triH)}, 
  \]
  and a final estimate follows from~\eqref{e:ehJerror} and Theorem~\ref{t:errorest} or~\ref{t:erroresthc} with $\bg$ replaced by $\Pi_J\bg$. 
\end{rem}

\section{Algorithms} \label{s:Algorithms}
In this section we give a practical guideline on how both ``versions'' of the method could be implemented.  We highlight the high-contrast case here, commenting on simplifications for the general case afterwards. 
Also, for simplicity, we consider $g\in\VVRM$, i.e., $g$ is piecewise constant; then $\tilde Tg=0$. If that is not the case and a pre-processing as described in Section~\ref{ss:findim} is desired, then compute $\Pi_J\bg$ as described in Lemma~\ref{l:exactappro}, and proceed with the computation replacing $\bg$ by $\Pi_J\bg$. This is particularly useful in the case of multiple right-hand sides. 

\begin{algorithm}
\caption{Localized Spectral Decomposition--LSD}
\label{algorithmlsd}
\begin{algorithmic}[1]
\State Generate a partition $\triH$ of $\OO$, and a partition $\faceh$ of the faces of elements in $\triH$
\State Define bases for $\VVRM$, $\LRM$, $\LTRM$ and $\tilde\Lambda_h^f$
\label{a:spacedef}
\State Build approximations $T_h$ for $T$ in~\eqref{e:definitionT}, and compute the action of $T_h$ on all basis functions of $\LRM$, $\LTRM$ and $\tilde\Lambda_h^f$ 
\label{a:Tapprox}
\State Compute generalized eigenvalues from~\eqref{e:eqeig1},~\eqref{e:eqeig2}. Define $\tilde\Lambda_h^\Pi$ and $\tilde\Lambda_h^\triangle$ by~\eqref{e:ltpmdef}, and set $\widetilde\Lambda_h^{0,\Pi}=\LTRM\oplus\tilde\Lambda_h^\Pi$
\label{a:eigen}
\State Find $\lambdarm\in\LRM$ from the first equation of~\eqref{e:mc3}
\label{a:loprblm}
\State Compute the action of $P^{\triangle,j}T$ for all basis functions of $\tilde\Lambda_h^{0,\Pi}$ from~\eqref{e:Pjdefhc}
\label{a:Pjdef}
\State Find $\tlambda_h^{0,\Pi,j}\in\widetilde\Lambda_h^{0,\Pi}$ from~\eqref{e:tlrmplusjeqtn}
\label{a:lamopij}
\State Find $\tilde\lambda_h^{\triangle,j}\in\tilde\Lambda_h^\triangle$ from~\eqref{e:lambdahminusj}
\State Find $u_h^{0,\text{LSD},j}\in\VVRM$ from~\eqref{e:uzerohigh} 
\State Compute $u_h^{\text{LSD},j}$ and $\bsigma_h^{\text{LSD},j}$ from~\eqref{e:uhjhighdef} and Remark~\ref{r:equil}
\end{algorithmic}
\end{algorithm}

\begin{rem}
The low-contrast case is a particular and simpler instance of the high-contrast one, and can be chosen by picking ${\alpha}_{\rm{stab}}>\alpha$. If however one is only interested in low-contrast cases then there is no need to compute eigenvalues, and the algorithm is simpler. Indeed, Step~\ref{a:eigen} of Algorithm~\ref{algorithmlsd} is no longer necessary, and one should set $\tilde\Lambda_h^\triangle=\tilde\Lambda_h^f$, $\widetilde\Lambda_h^{0,\Pi}=\LTRM$.  Also, in Step~\ref{a:Pjdef} let $P^{\triangle,j}={P}^j$ from~\eqref{e:Pjdef}. 
\end{rem}

Note that the sub-problems in Algorithm~\ref{algorithmlsd} are either local, semi-local or global. The problems of Step~\ref{a:Tapprox} are local and can be solved in parallel. The eigenvalue problems of Step~\ref{a:eigen} are also local since each face $F$ belongs to at most two elements. The computation of $\lambdarm$ in Step~\ref{a:loprblm} involves a system that depends on $H$ only. Step~\ref{a:Pjdef} requires solving semi-local problems, since it involves solving systems~\eqref{e:PFjdefhc}, posed on local patches of elements. The computation of Step~\ref{a:lamopij} requires solving a global problem with dimension of $\tilde\Lambda_h^{0,\Pi}$. There are no matrix inversions for the remaining steps of the algorithm, since all the operators were pre-computed at this point.

\section{Numerical Examples}
In this section we consider numerical tests for three different problems. The first one involves a smooth problem with analytical solution, allowing exact error computation. In the second and third problem this is no longer possible, and we compare solutions obtained by the LOD or LSD methods with the exact
discrete solution obtained by solving the whole system~\eqref{e:weak-hybridh}. In the tests that follow (with the exception of the high-contrast channel Test~\ref{ss:hcchannels}), we use the following notation and data: 
\begin{itemize}
\item $H=1/8$ is the maximum size of macro-elements
\item $N_e=8$ is the dimension of $\tildeLambda_h^F$ on each face $F$
\item $h\approx H/9$ is the maximum size of micro-elements. 
\end{itemize}
For all tests, $\Omega=(0,1)\times(0,1)$, and the corresponding mesh and sub-mesh with the above parameters is as in Figure~\ref{f:mesh} (the exception is Table~\ref{t:lsdbeanbag2}). 
\begin{figure}
\centering
\epsfig{file=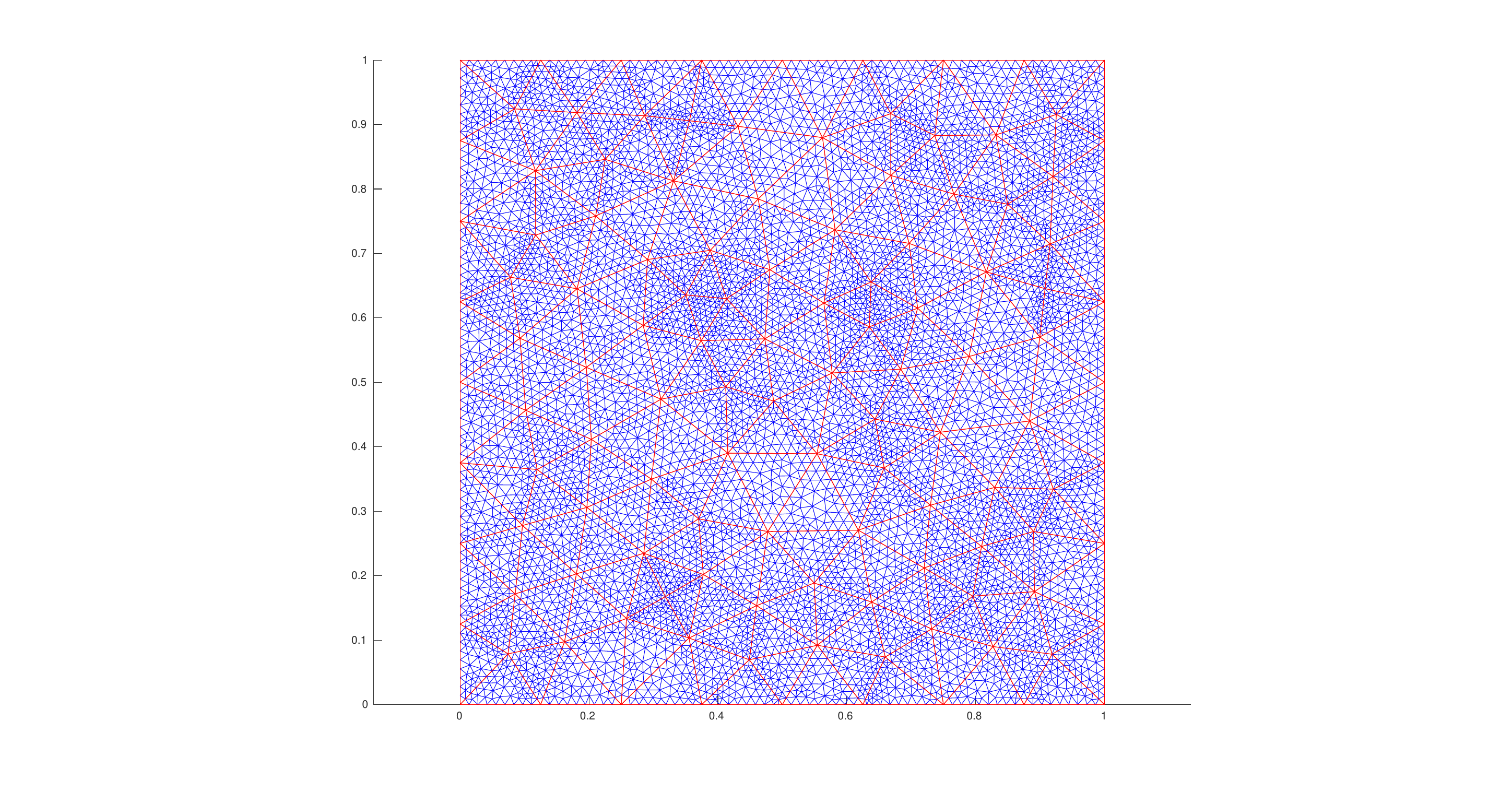, scale=0.3}
\caption{Mesh and submesh employed in all numerical tests. We consider $H=1/8$ and $h\approx H/9$.} 
\label{f:mesh}
\end{figure}

The acronyms NLOD stands for \emph{Non-local Orthogonal Decomposition} and corresponds to the case where the system~\eqref{e:weak-hybridh} is fully solved and no localization techniques are applied. 

\subsection{Exact Solution available}~\label{ss:exact}
Let $\A(x,y)=1$ and the exact solution $u(x,y)=x(x-1)y(y-1)$. For such kernel, we display in Figure~\ref{f:an_sol} the energy decay of Theorem~\ref{t:decaythmhc} if $\tilde\Lambda_h^\Pi$ is empty. 
\begin{figure}
\centering
\epsfig{file=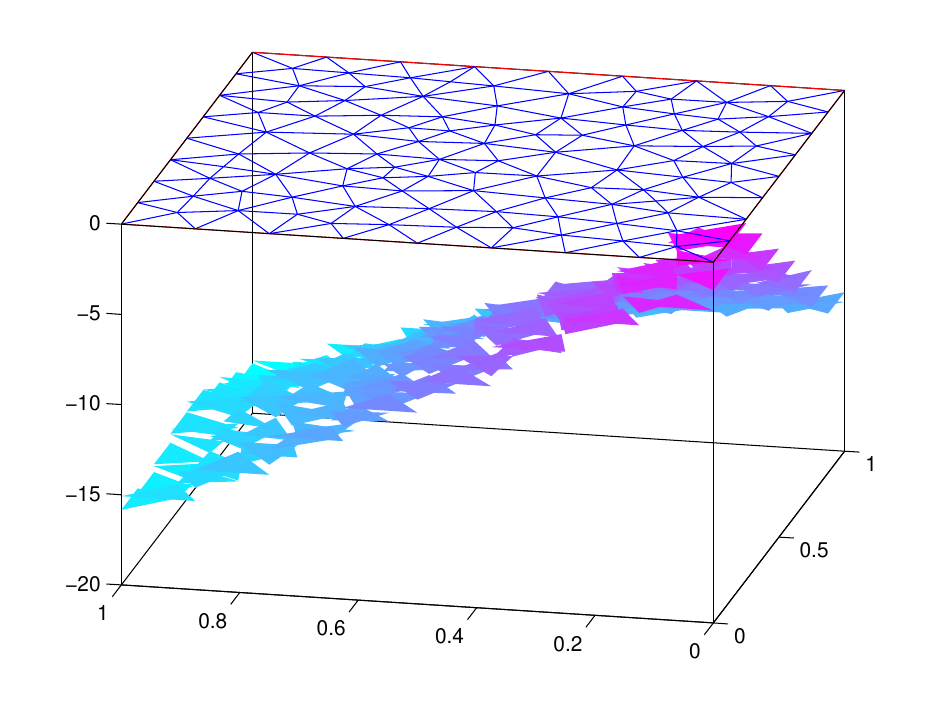, height=6cm, width=7.5cm}
\caption{Test~\ref{ss:exact}: Log-plot of energy decay in the case of empty $\tilde\Lambda_h^\Pi$.} 
\label{f:an_sol}
\end{figure}
The various relative errors with respect to the exact solution $u$ in the energy norm follow in Tables~\ref{t:lodanalyt} and~\ref{t:lsdanalyt}. In Table~\ref{t:lodanalyt} we compute the relative error for the NLOD, and for the the LOD with various localizations. As expected the error increases as one reduces the patch size where basis functions are computed. For Table~\ref{t:lsdanalyt} the analysis is more complex since both the localization \emph{and} $\astab$ play important roles. For a small $\astab=1.1$ the dimension of $\tilde\Lambda_h^\Pi$ (=1260) becomes large compared the the total dimension of $\tilde\Lambda_h^f$ (=2624) and the errors are small and comparable to the NLOD. For a large $\astab=3.0$ however, the dimension of $\tilde\Lambda_h^\Pi$ is relatively small (=634), that is, about two eigenvectors per each macro-element edge, and the errors are similar to the NLOD if the patches are not too small. These two eigenvectors are responsible for eliminating the singularities at the macro-elements corners that are found in the LOD method; see the two dominating eigenvalues in Figure~\ref{f:eig_osc}.

For $J=1$, there is a larger error, but the method is better than LOD (cf. Table~\ref{t:lodanalyt}). 
\begin{table}
 \centering \small
 \begin{tabular}{|l|l|}\hline
       & Relative energy error \\\hline
  NLOD & 0.0025 \\\hline
  LOD (J=3) & 0.0095 \\\hline
  LOD (J=2) & 0.0385 \\\hline
  LOD (J=1) & 0.0683 \\\hline
    \end{tabular}
\caption{Test~\ref{ss:exact}: LOD version with no localization (NLOD), and different localizations (J=1,2,3).}
\label{t:lodanalyt}
\end{table}

\begin{table}
\centering \small
\begin{tabular}{|l|l|l|}\hline
&\multicolumn{2}{c|}{Relative energy error}\\\hline
&$\astab=1.1$ (dim $\tildeLambda_h^\Pi=1260$)&$\astab=3.0$ (dim $\tildeLambda_h^\Pi=634$)\\\hline
LSD (J=3) & 0.0025 & 0.0025\\\hline
LSD (J=2) & 0.0025 & 0.0026\\\hline
LSD (J=1) & 0.0035 & 0.0229\\\hline
    \end{tabular}
\caption{Test~\ref{ss:exact}: Relative energy error for LSD versions with different localizations (J=1,2,3). The total dimension of $\tildeLambda_h^f$ is 2624.}
\label{t:lsdanalyt}
\end{table}

\subsection{Oscillatory kernel}\label{ss:osc}
Consider $f=1$ and the oscillatory kernel
\[
\A(x,y)=\frac{\beta-\alpha}2\bigl(1+\sin(2\pi x/\epsilon)\sin(2\pi y/\epsilon)\bigr)+\alpha, 
\]
where $\alpha=0.01$, $\beta=10$ and $\epsilon=1/16$.

In this example, we also test the use of \emph{incomplete LU Decomposition} of $j$-th fill-in to invert $P^\triangle$ instead of using $P^{\triangle,j}$ in the case of LSD, and to invert $P$ instead of using $P^j$ in the case of LOD. We name the above schemes \emph{iLSD} and \emph{iLOD}. As we show below, the results are encouraging, albeit the lack of theory.  

The eigenvalues of problems I (equation~\eqref{e:eqeig1}) for a fixed edge are plotted in Figure~\ref{f:eig_osc}. Note the clustering of the eigenvalues at 1. Then, Table~\ref{t:lodosc} presents the results for the LOD and iLOD (with incomplete LU) versions, with different localization levels. Finally, Tables~\ref{t:lsdosc} present results for LSD and iLSD (with incomplete LU) with different values of $\astab$. 
\begin{figure}
\centering
\epsfig{file=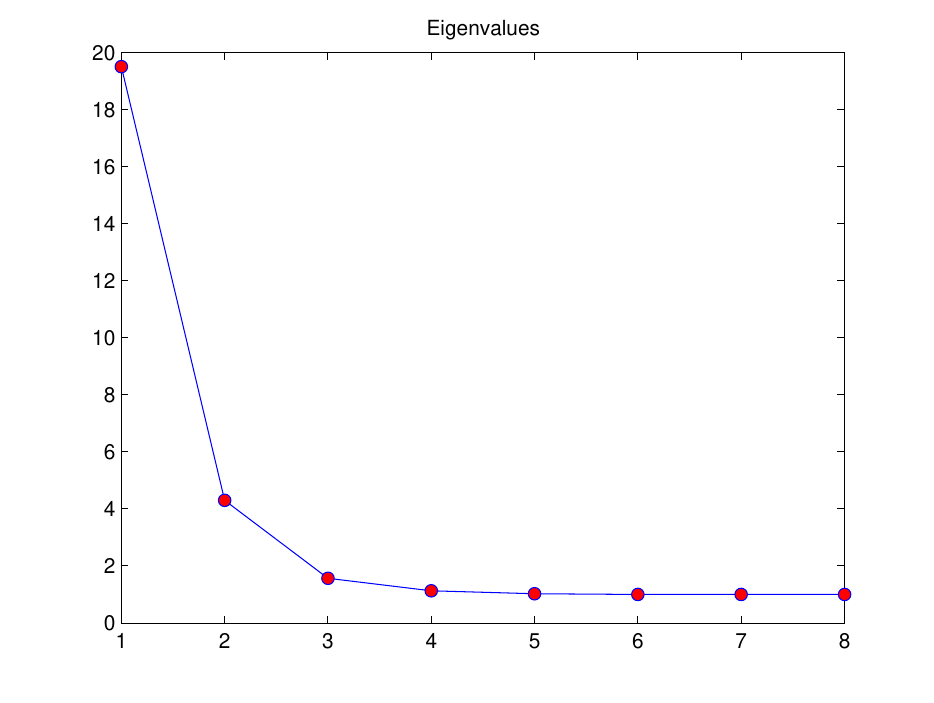, height=6cm, width=7.5cm}
\caption{Test~\ref{ss:osc}: Eigenvalues for problems I, II.} 
\label{f:eig_osc}
\end{figure}
\begin{table}
 \centering \small
 \begin{tabular}{|l|l|l|l|}\hline
       &  LOD&iLOD\\\hline
  J=3 &  0.0092 & 0.0040\\\hline
  J=2 &  0.0499 & 0.0071\\\hline
  J=1 &  0.0778 & 0.0186\\\hline
    \end{tabular}
\caption{Test~\ref{ss:osc}: Relative energy error for LOD and LOD with incomplete LU versions, with different localizations (J=1,2,3)}
\label{t:lodosc}
\end{table}
\begin{table}
 \centering \small
 \begin{tabular}{|l|l|l|l|}\hline
       &$\astab=1.1$&$\astab=3.0$ \\\hline
  LSD (J=3) &  5.7080e-11& 2.7741e-06\\\hline
  LSD (J=2) &  3.2054e-06& 7.0387e-04\\\hline
  LSD (J=1) &  0.0032& 0.0291\\\hline
 \end{tabular}
 \qquad \qquad
 \begin{tabular}{|l|l|l|l|}\hline
       &$\astab=1.1$&$\astab=3.0$ \\\hline
  iLSD (J=3) &  1.1617e-08& 1.7775e-05\\\hline
  iLSD (J=2) &  2.3785e-07& 8.1785e-05\\\hline
  iLSD (J=1) &  7.6412e-06& 6.1034e-04\\\hline
    \end{tabular}
\caption{Test~\ref{ss:osc}: Relative energy error for LSD versions with different localizations (J=1,2,3) using LSD and LSD with incomplete LU}
\label{t:lsdosc}
\end{table}

\subsection{Beanbag test}\label{ss:beabag}
Consider $\Omega=(0,1)\times(0,1)$, constant $f=1$ and 
\[
\A(x,y)=\begin{cases}1&\text{if }|x-0.5|+|y-0.25|<0.3,\\
                     \amin&\text{otherwise}.\end{cases}
\]
A refined exact solution and some approximations are depicted in Figure~\ref{f:beanbag_exact}, for $\amin=10^{-4}$. The figure on the left displays the ``exact'' solution, computed by solving~\eqref{e:weak-hybridh} without localization. On the top right we display the LOD solution for $J=1$. On the bottom figures, two solutions for $J=1$ are shown. On the left figure $\astab=3.0$, and on the right one $\astab=1.7$. 
\begin{figure}
\centering
\epsfig{file=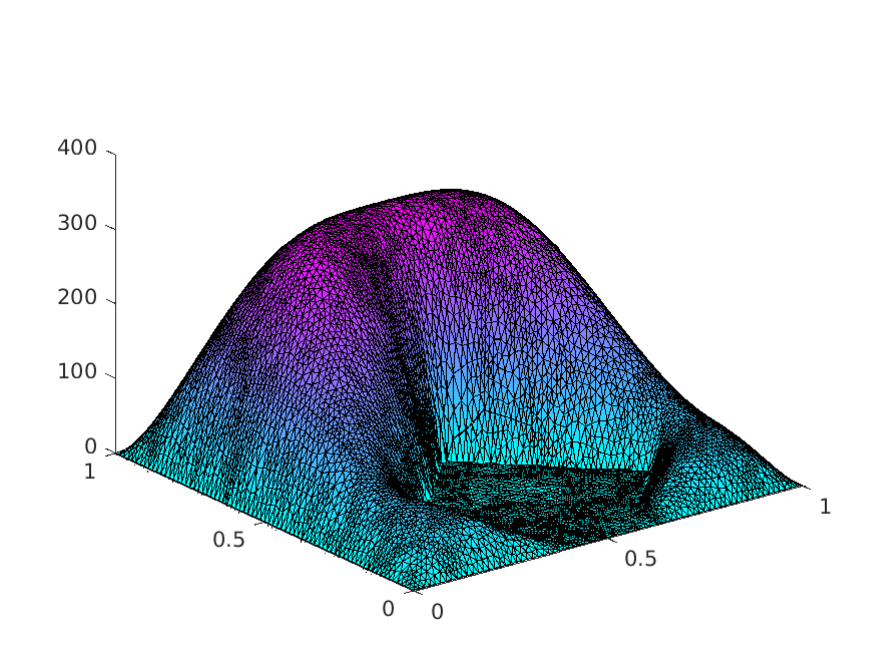, height=6cm, width=7.5cm}
\epsfig{file=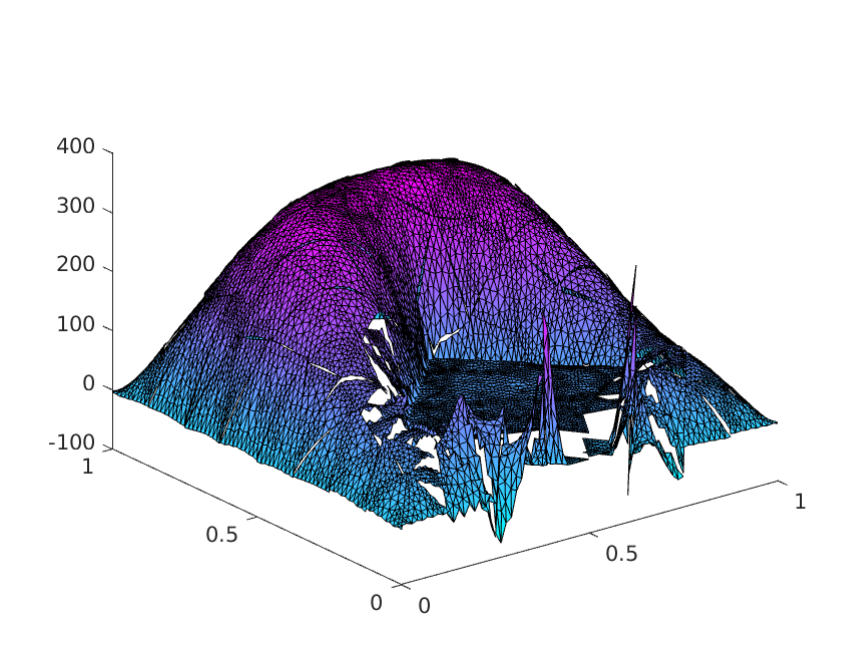, height=6cm, width=7.5cm}
\epsfig{file=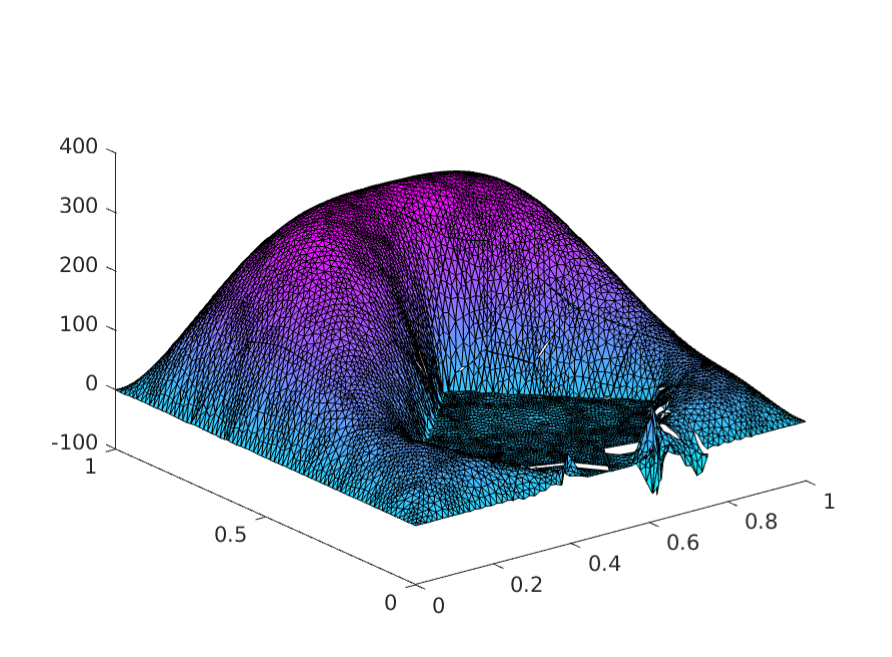, height=6cm, width=7.5cm}
\epsfig{file=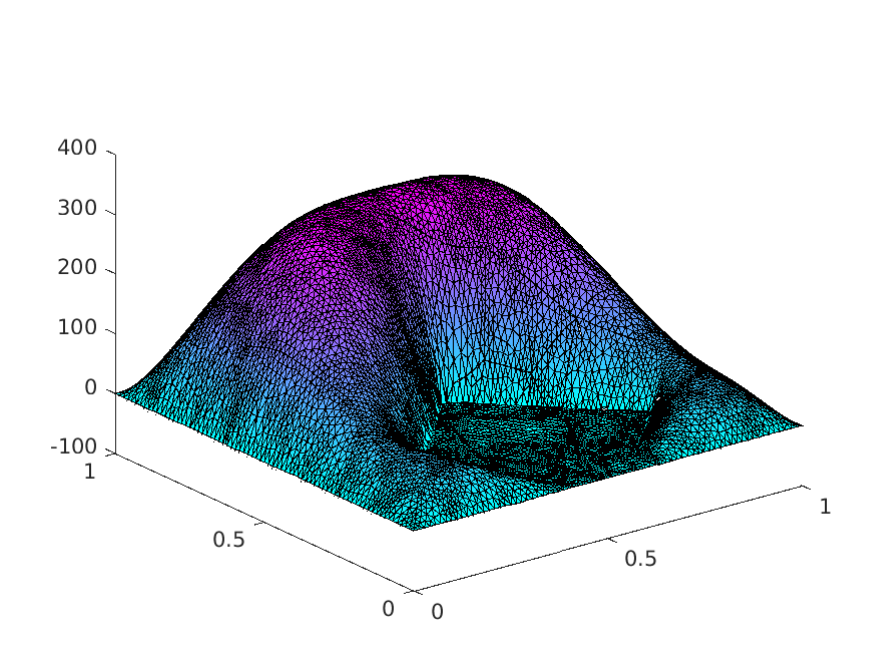, height=6cm, width=7.5cm}
\caption{Test~\ref{ss:beabag}. Top: Exact solution and LOD with $J=1$. Bottom: LSD solutions  with $J=1$ and $\astab=3.0$ and $\astab=1.7$.} 
\label{f:beanbag_exact}
\end{figure}
Next, Table~\ref{t:lodbeanbag} compares the approximation results using LOD for different localization levels, and Table~\ref{t:lsdbeanbag} presents the results for LSD with different localization levels and different $\astab$. The final Table~\ref{t:lsdbeanbag2} is also interesting, the number of eigenvalues per macro-element edge is not very sensitive with respect to $N_e$, it depends more on the coefficient $\A$.

\begin{table}
 \centering \small
 \begin{tabular}{|l|l|l|}\hline
       & Relative energy error\\\hline
  LOD (J=3) &  0.0616\\\hline
  LOD (J=2) &  0.1008\\\hline
  LOD (J=1) &  0.5289\\\hline
    \end{tabular}
\caption{Test~\ref{ss:beabag}: LOD version with different localizations (J=1,2,3)}
\label{t:lodbeanbag}
\end{table}
\begin{table}
 \centering \small
 \begin{tabular}{|l|l|l|l|l|}\hline
&\multicolumn{3}{c|}{Relative energy error}\\\hline
       &$\astab=1.1$ (dim $\tildeLambda_h^\Pi=1250$)&$\astab=1.7$ (dim $\tildeLambda_h^\Pi=734$) &$\astab=3.0$ (dim $\tildeLambda_h^\Pi=620$) \\\hline
  LSD (J=3) &  1.4477e-10&1.1590e-06& 1.1479e-05\\\hline
  LSD (J=2) &  5.4182e-05&0.0022& 0.0130\\\hline
  LSD (J=1) &  0.0274&0.0346& 0.2211\\\hline
    \end{tabular}
\caption{Test~\ref{ss:beabag}: Relative energy error for LSD versions with different localizations (J=1,2,3). The total dimension of the eigenspace $\tildeLambda_h^f$ is 2624}
\label{t:lsdbeanbag}
\end{table}

\begin{table}
 \centering \small
 \begin{tabular}{|l|l|l|l|l|l|l|l|}\hline
H&$N_e$&dim $\tildeLambda_h^f$&\multicolumn{3}{c|}{Relative energy error (dim $\tildeLambda_h^\Pi$)}
\\\hline
\multicolumn{3}{|c|}{}&$\astab=1.1$ &$\astab=1.7$  &$\astab=3.0$
\\\hline
$1/8$& $8$& 2624 &5.4182e-05 (1250)&0.0022 (734)& 0.0130 (620) \\\hline
$1/8$& $4$& 1312 & 3.4172e-05 (975)& 0.0063 (644)& 0.0072 (445) 
\\\hline
$1/16$& $8$& 10720 & 1.1227e-05 (5170)& 0.0016 (2948)& 0.0111 (2604)
\\\hline
$1/16$& $4$& 5360 & 3.0627e-05 (3991)& 0.0011 (2663)& 0.0122 (1869)
\\\hline
\end{tabular}
\caption{Test~\ref{ss:beabag}: Test with $H=1/16$ and $J=2$. Relative energy error for LSD versions with localization $J=2$. }
\label{t:lsdbeanbag2}
\end{table}

Finally, Figure~\ref{f:beanbag_contrast} shows how the relative energy error behaves with LOD (in blue) and LSD (in red) as the contrast (=$1/\amin$) increases. 
\begin{figure}
\centering
\epsfig{file=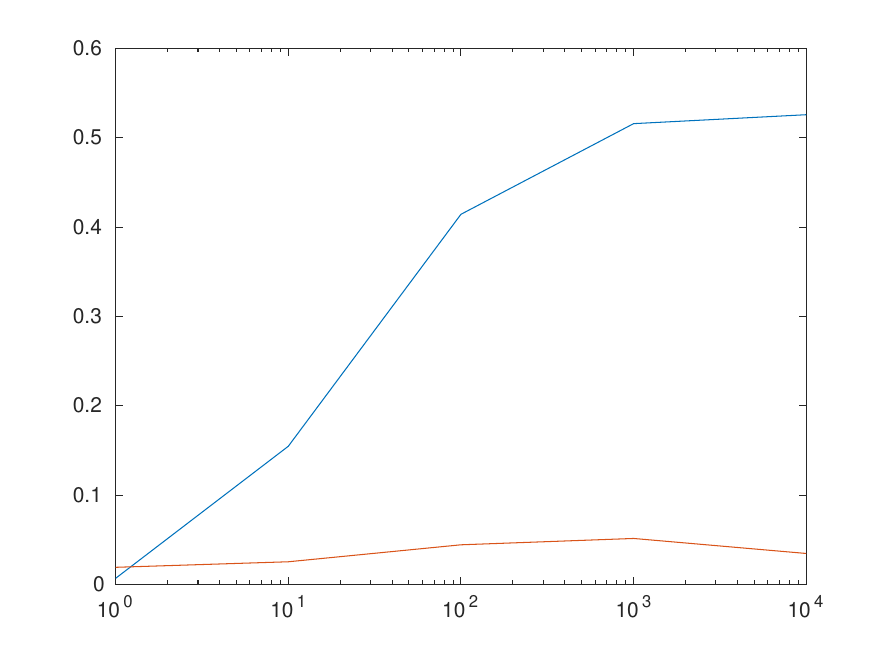, height=6cm, width=7.5cm}
\caption{Test~\ref{ss:beabag}: Relative energy error with respect to contrast (=$1/\amin$). The blue curve corresponds to LOD with $J=1$ and the red curve is related to the LSD method with $J=1$ and $\astab=1.7$} 
\label{f:beanbag_contrast}
\end{figure}

\subsection{High-contrast channels}\label{ss:hcchannels}
In this example we consider the macro element diameter $H=1/8$, the second-level mesh diameter $h=H/17$, and the dimension of $\tildeLambda_h^F$ on each face $F$ as $N_e=16$. For a given positive number $a$, the kernel $\A$ is defined by
\begin{equation}\label{e:coeffchann}
\A(x,y)=\begin{cases}
a&\text{if }\|(x-0.5,y-0.5)\|_2<1/40\text{ and }y>1/2,\\
a&\text{if }|y-0.6|<1/40,\\
a^{-1}&\text{otherwise.}
\end{cases}
\end{equation}
The level curve defined by $\A$ is depicted in Figure~\ref{f:channelcoeff}. 
\begin{figure}
\centering
\epsfig{file=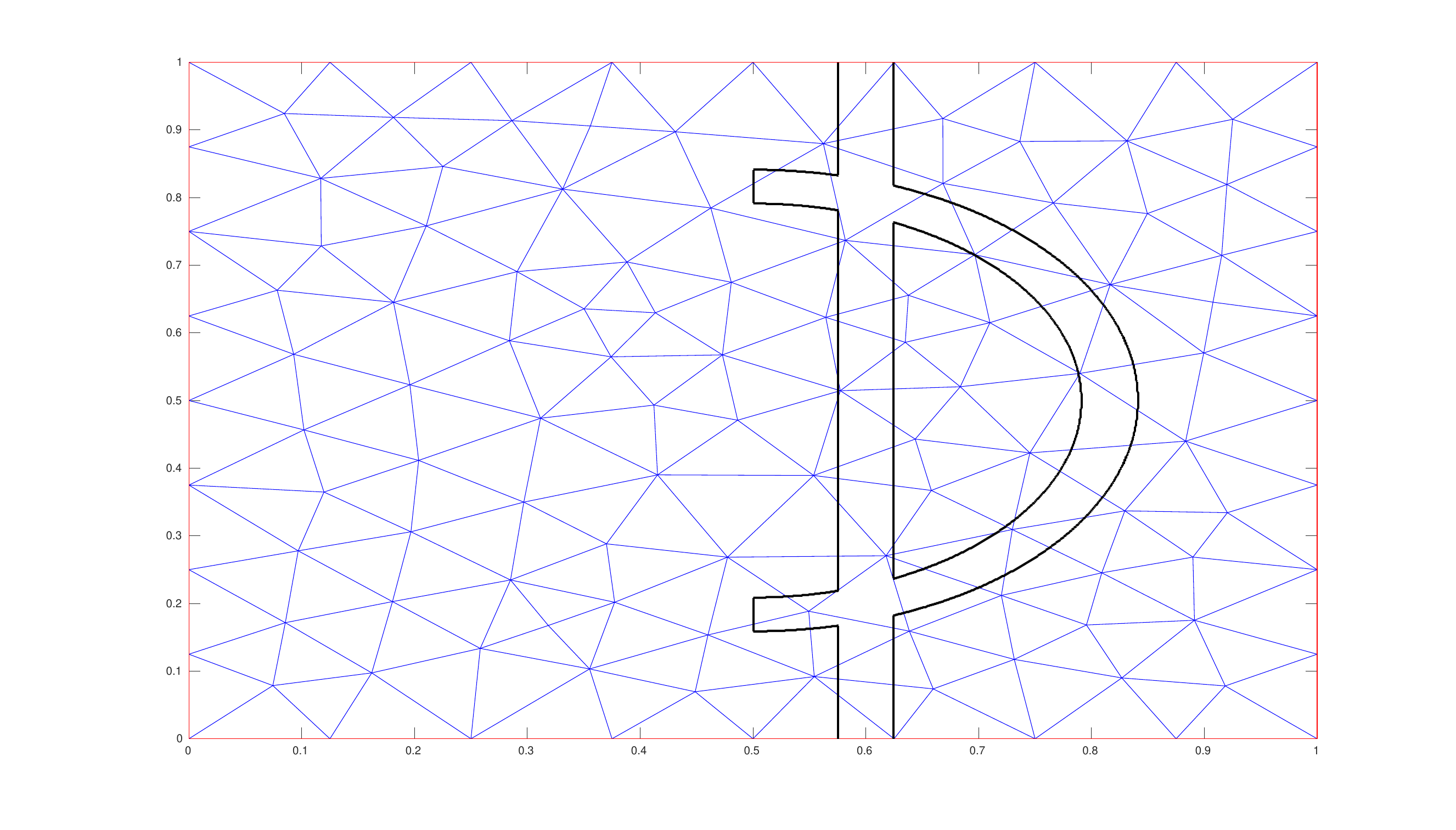, scale=0.15}
\epsfig{file=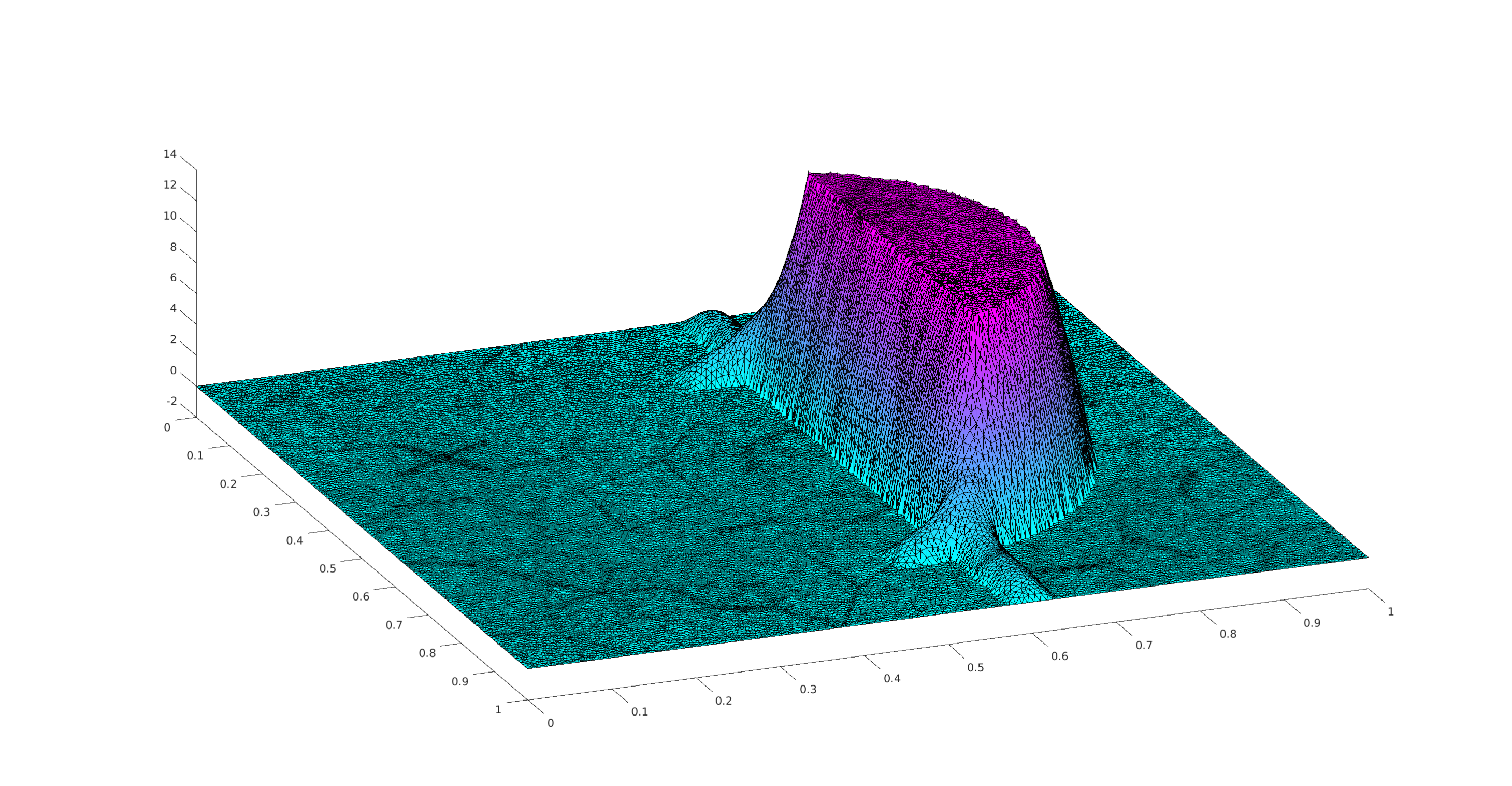, scale=0.19}
\caption{The left figure displays the coarse mesh and the level curve of the coefficient~\eqref{e:coeffchann}. Basically, the geometry is defined by a channel that bifurcates, forming an ``island.'' The coefficient is equal to $a^{-1}$ inside the channel and $a$ on the ``plateau'' and in the island. On the right we present the ``exact'' solution for $a=10^{7/2}$. The total contrast is $10^7$.}
\label{f:channelcoeff}
\end{figure}

We consider a sequence of problems defined  by increasing values of $a$. Note that the contrast $\kappa=a^2$. We test the performance of the LOD and LSD methods, presenting the results in Table~\ref{t:channel}.  We test the LOD method with $J=4$, and it is clear that the performance deteriorates as the contrast increases. For the LSD method, we set $J=2$ and present results for various values of $\astab$. We remark that that the full  system~\eqref{e:weak-hybridh} has $5248$ equations, while the LSD final systems range from 1151 ($\astab=1.3$) to 632 equations ($\astab=3.0$). 

\begin{table}
\centering \small
\begin{tabular}{|l|l|l|l|l|l|l|}\hline
\multicolumn{2}{|c|}{}&\multicolumn{5}{c|}{Relative energy error}
\\\hline
\multicolumn{2}{|c|}{}&LOD ($J=4$) &\multicolumn{4}{c|}{LSD ($J=2$)}
\\\hline
$a$& $\kappa$&& $\astab=1.3$&$\astab=1.4$&$\astab=1.5$ &$\astab=3.0$
\\\hline
$10^2$ & $10^4$ &$0.002$ &$0.002$  &$0.003$ &$0.003$ &$0.025$
\\\hline
$10^{5/2}$ &$10^5$ &$0.005$ &$0.004$ &$0.005$ &$0.006$ &$0.021$
\\\hline
$10^3$ &$10^6$ &$0.015$ &$0.002$ &$0.005$ &$0.005$ &$0.021$
\\\hline
$10^{7/2}$ &$10^7$ &$0.042$ &$0.003$ &$0.006$ &$0.006$ &$0.021$
\\\hline
\end{tabular}
\caption{Test~\ref{ss:hcchannels}: First and second columns give the value of $a$ and the contrast $\kappa$; third column lists relative energy error with LOD method with $J=4$; Fourth-seventh columns shows relative energy error with  LSD scheme with $J=2$ and increasing values of $\astab$.}
\label{t:channel}
\end{table}

\appendix
\section{Auxiliary results} \label{s:Auxiliaryresults}
\begin{lemma} \label{l:lrmsystem0}
Let  $\lambdarm_i$ defined as in~\eqref{e:lrmdef}. Then  $\{\lambdarm_i\}_{i=1}^N$ are linearly independent.   
\end{lemma}
\begin{proof}
  Assume that there exist constants $\beta_1,\dots,\beta_N$ such that $\sum_{i=1}^N\beta_i\lambdarm_i=0$. Consider $\tau_i$, $\tau_j$ two adjacent triangles sharing a face $F$ and let $v\in \VH$ with support in $\tau_i$
  such that $v|_{F'}=0$ if $F'\ne F$ and $\int_Fv=1$. Then
  \[
  0=\sum_{i=1}^N\beta_i(\lambdarm_i,v)_{\btriH}=\beta_i+\beta_j.
  \]
  If $F\subset\partial\Omega\cap\partial \tau_i$ then using the same
  arguments we have $\beta_i=0$. Then $\beta_i=0$ for all $i=1,\dots,N$ by using the connectivity of $\triH$.  
  \end{proof}
\begin{lemma}\label{l:lrmsystem}
  Let $\mu\in\Lh$. Then the problem of finding $\murm$ such that $(\murm,\vvrm)_{\btriH}=(\mu,\vvrm)_{\btriH}$ for all $\vvrm\in\VVRM$ is well-posed. 
\end{lemma}
\begin{proof}
  Since we are dealing with finite dimensional spaces, it is enough to prove that the $N\times N$ matrix with components $(\lambdarm_i,\vvrm_j)_{\btriH}$ is non-singular, where $\vvrm_i$ is the characteristic function of the element $\tau_i$. Note that it follows from the definition of $\lambdarm_i$ that
\[
|(\lambdarm_i,\vvrm_j)_{\btriH}|=
\begin{cases}|\partial \tau_i|&\text{if }i=j,
  \\
|\partial \tau_i\cap\partial \tau_j|&\text{otherwise},
\end{cases}
\]
and then the matrix is diagonally dominant. Consider now an element $\tau_i$ such that $\partial \tau_i\cap\partial\Omega\ne\emptyset$. Then,
\[
|(\lambdarm_i,\vvrm_i)_{\btriH}|=|\partial \tau_i|>\sum_{j\ne i}|\partial \tau_i\cap\partial \tau_j|.
\]
We note that the matrix is irreducible since given any two distinct elements 
$\tau_i$ and $\tau_k$, there exists a path of adjacent
faces connecting  $\tau_i$ to  $\tau_k$. Then the matrix is irreducibly diagonally dominant, and from~\cite{MR2093409}*{Theorem 1.11}, it is non-singular.
\end{proof}

\begin{lemma}\label{l:extensions}
Let $\mu\in\Lambda_h$ and $\tau\in\triH$, assume that~\eqref{e:bounds} holds, and denote by $T_\I$ the harmonic extension defined by~\eqref{e:definitionT} replacing $\A$ by the identity operator. It follows then that 
  \[
  \frac{1}{a_{\min}^\tau}|T_\I\mu|_{H^1(\tau)}^2
  \ge|T\mu|^2_{H_\A^1(\tau)}
  \ge\frac{1}{a_{\max}^\tau}|T_\I\mu|_{H^1(\tau)}^2. 
  \]
\end{lemma}
\begin{proof}
  First note that, for any non-vanishing $\lambda$,
  \begin{multline*}
    \inf_{v\in\widetilde H^1(\tau)}\frac12(\lambda\bgrad v,\bgrad v)_\tau-(\mu,v)_{\partial\tau}
    =\inf_{v \in\widetilde H^1(\tau)}\frac12(\lambda\bgrad\lambda^{-1}v,\bgrad\lambda^{-1}v)_\tau
    -(\mu,\lambda^{-1}v)_{\partial\tau}
  \\
  =\frac1\lambda\inf_{v\in\widetilde H^1(\tau)}\frac12(\bgrad v,\bgrad v)_\tau
  -(\mu,v)_{\partial \tau}. 
  \end{multline*}
  Then 
  \begin{multline*}
  -\frac12(\mu,T\mu)_{\partial\tau}
  =\inf_{v \in\widetilde H^1(\tau)}\frac12(\A\bgrad v,\bgrad v)_\tau-(\mu,v)_{\partial \tau}
  \le\inf_{v \in\widetilde H^1(\tau)}a_{\max}^\tau\frac12(\bgrad v,\bgrad v)_\tau
  -(\mu,v)_{\partial \tau}
  \\
  =-\frac1{2 a_{\max}^\tau}(\mu,T_\I\mu)_{\partial\tau}. 
  \end{multline*}
  Similarly, 
  \begin{equation*}
  -\frac12(\mu,T\mu)_{\partial\tau}
  \ge\inf_{v\in\widetilde H^1(\tau)}a_{\min}^\tau\frac12(\bgrad v,\bgrad v)_\tau-(\mu,v)_{\partial\tau}
  =-\frac1{2a_{\min}^\tau}(\mu,T_\I\mu)_{\partial\tau}. 
  \end{equation*}
\end{proof}
\section{Notations} \label{s:Notations}
Inner products, dualities and norms:
\begin{itemize}
\item $(\cdot,\cdot)_{\triH}$, $(\cdot,\cdot)_{\btriH}$, $(\cdot,\cdot)_{\partial\tau}$: inner and duality products; Section~\ref{ss:dhc}
\item $\|\cdot\|_{\HdivA}$, $\|\cdot\|_{H_\A^1(\triH)}$, $\|\cdot\|_{H_\A^{-1/2}(\triH)}$: equation~\eqref{e:norms}
\end{itemize}

Functional spaces:
\begin{itemize}
\item $\VH$ and $\LH$: piecewise $H^1$ functions, and trace of $\Hdiv$ functions; equation~\eqref{e:spaces}
\item $\VVRM$, $\tildeVV$, $\widetilde H^1(\tau)$: piecewise constants and zero average functions; equation~\eqref{e:poHtdef}
\item $\mathcal F_J$: eigenspace generating a finite dimension space for the right-hand side; equation~\eqref{e:Fjdef}
\item $\Lh$: piecewise constants on the refined skeleton; equation~\eqref{e:lambdaconst}
\item $\LRM$: constants flux traces on element boundaries; equation~\eqref{e:lammbdadecomp}
\item $\tildeLambda_h=\VVRM^\perp$: zero average flux traces on element boundaries; equation~\eqref{e:lammbdadecomp}
\item $\LTRM$: constants on faces with zero average on element boundaries; equation~\eqref{e:lambdafaces}
\item $\tildeLambda_h^f$: zero average flux traces on faces; equation~\eqref{e:lambdafaces}
\item$\tilde\Lambda_h^{f,\tau,j}\subset\tilde\Lambda^f_h$: equation~\eqref{e:operPkfdef} 
\item $\tilde\Lambda_h^\tau$: local space; equation~\eqref{e:hcspaces}
\item $\tilde\Lambda_h^F$, $\tilde\Lambda_h^{F_\tau^c}$: restriction of $\tilde\Lambda_h^f$ to $F$, $F_\tau^c$; equation~\eqref{e:hcspaces}
\item $\tilde\Lambda_h^{F,\triangle}$, $\tilde\Lambda_h^{F,\Pi}$, $\tilde\Lambda_h^\Pi$, $\tilde\Lambda_h^\triangle$: spectral spaces; equations~\eqref{e:Lpmdef},~\eqref{e:ltpmdef}
\item $\widetilde\Lambda_h^{0,\Pi}:=\LTRM\oplus\tilde\Lambda_h^\Pi$: replaces $\LTRM$ in the high-contrast case; equation~\eqref{e:Ltopidef}
\end{itemize}

Operators
\begin{itemize}
\item $P:\VH\to\tilde\Lambda_h^f$: non-local operator; equation~\eqref{e:Pdef}
\item $P^{K,j}$, $\tilde{P}^j$, $P^{F,j}$, $P^j$: equations~\eqref{e:operPkfdef}, \eqref{e:PKjdef}, \eqref{e:PFjdef}, \eqref{e:Pjdef}
\item $P^\triangle$, $P^{\triangle,K,j}$, $\tilde{P}^{\triangle,j}$, $P^{\triangle,F,j}$, $P^{\triangle,j}$: equations~\eqref{e:Pdefhigh}, \eqref{e:Ptriagkj}, \eqref{e:Pjdefminus}, \eqref{e:PFjdefhc}, \eqref{e:Pjdefhc}
\item $\Pi_J$: $L^2(\Omega)$ orthogonal projection on $\mathcal F_{J}$: Section~\ref{ss:findim}
\item $T$, $\tilde T$, $T_{FF}^\tau$, $T_{FF^c}^\tau$, $T_{F^cF}^\tau$, $T_{F^cF^c}^\tau$, $\widehat T_{FF}^\tau$: local operators; Sections~\ref{Hybrid} and~\ref{ss:dhc} 
\end{itemize}

Unknowns:
\begin{itemize}
\item $\lambda$: trace of the elementwise fluxes; equation~\eqref{e:weak-hybridH}
\item $\lambda_h$: traces of ``surrogate'' flux; equation~\eqref{e:weak-hybridh}
\item $\lambdarm$, $\tlambdarm_h$, $\tilde\lambda_h^f$: decompose $\lambda_h$; equation~\eqref{e:mc3}
\item $\lambda_h^j$, $\tlambdarmj_h$, $\tilde\lambda_h^{f,j}$: solutions using the $P^j$ operator; equations~\eqref{e:uhjdef} and~\eqref{e:lambdahfj}
\item $\tlambda_h^{0,\Pi}$, $\tilde\lambda_h^\triangle$: components of the solution using the $P^\triangle$ operator; equations~\eqref{e:Pdefhigh},~\eqref{e:alternative}
\item $\tlambda_h^{0,\Pi,j}$, $\tilde\lambda_h^{\triangle,j}$: components of the solution using the $P^{\triangle,j}$ operator; equation~\eqref{e:uhjhighdef}
\item $\bsigma$: flux; equation~\eqref{e:disp-stress}
\item $\bsigma_h$: ``surrogate'' flux; equation~\eqref{e:disp-stressh}
\item $\bsigma_h^{\text{LSD},j}$: LSD flux; Remark~\ref{r:equil}
\item $u$: solution of the original problem; equation~\eqref{e:elliptic}
\item $\uurm$, $\tilde u$: average and zero average components of $u$; equations~\eqref{e:weak-e1} and~\eqref{e:weak-e3}
\item $u_h$: solution of the ``surrogate'' discrete problem; equation~\eqref{e:disp-stressh}
\item $\uurm_h$: average of $u_h$; equation~\eqref{e:weak-hybridh}
\item $u_h^j$, $u_h^{0,j}$: solutions using the $P^j$ operator; equation~\eqref{e:uhjdef}
\item $u_h^{\text{LSD},j}$, $u^{0,\text{LSD},j}_h$: solutions using the $P^{\triangle,j}$ operator; equation~\eqref{e:uhjhighdef}
  \item $v_1,\,v_2,\,v_3,\,\dots$: eigenfunctions generating a finite dimension space for the right-hand side; equation~\eqref{e:eigenproblem}
\end{itemize}

Other notations:
\begin{itemize}
\item $\amin$, $\amax$, $a_-$, $a_+$, $\amin^\tau$, $\amax^\tau$: bounds for the eigenvalues of $\A$; equation~\eqref{e:bounds} and Lemma~\ref{l:eigenbound}
\item $\A$: symmetric coefficients tensor; equation~\eqref{e:elliptic}
\item $\alpha_{\rm{stab}}$: controls the decay rate of the solutions to the non-local problems; equation~\eqref{e:Lpmdef}
\item $c_P$, $C_{P,G}$: local and global weighted Poincar\'e inequality constants; Remark~\ref{r:sharpesthc}
\item $\beta_{H/h}=1+\log(H/h)$: Lemma~\ref{l:eigenbound}
\item $d$: dimension; equation~\eqref{e:elliptic}
\item $\faceh$: partition of the faces of elements in $\triH$; Section~\ref{s:Hlocal}
\item $F_\tau^c=\partial\tau\backslash F$: $\partial\tau$ except the face $F$; above equation~\eqref{e:hcspaces}
\item $g$: right-hand side of the original problem; equation~\eqref{e:elliptic}
\item $\gamma$: constant depending on shape regularity of $\triH$; Lemma~\ref{l:eigenbound}
\item $H$, $h$: coarse and fine mesh characteristic lengths
\item $\HH$:  the method's ``target precision''; equation~\eqref{e:HHdef}
\item $\kappa=\max_{\tau\in\triH}\kappa^\tau$, $\kappa^\tau=\amax^\tau/\amin^\tau$: Lemma~\ref{l:eigenbound}
\item $\nn^\tau$: unit size normal vector pointing outward $\tau$; Section~\ref{Hybrid}
\item $\Omega$, $\dO$: bi- or tri-dimensional domain and its boundary; Section~\ref{s:intro}
\item $\tau$, $\partial\tau$: typical element in $\triH$ and its boundary
\item $\triH$: partition of $\Omega$; Section~\ref{Hybrid}
\item $\T_j(K)$: equation~\eqref{e:taujk}
\end{itemize}


\end{document}